\documentclass{amsart}
\usepackage{amsmath}
\usepackage{graphicx}
\usepackage{latexsym}
\usepackage{color}
\usepackage{amscd}
\usepackage[all]{xy}
\usepackage{enumerate}
\usepackage{hyperref}
\usepackage{subfigure}
\usepackage{soul}
\usepackage{comment}
\usepackage{epsfig}
\usepackage{epstopdf}
\newtheorem{thm}{Theorem}

\newtheorem{cor}[thm]{Corollary}

\newtheorem{theorem}[thm]{Theorem}
\newtheorem{lemma}[thm]{Lemma}

\newtheorem{proposition}[thm]{Proposition}

\theoremstyle{definition}

\newtheorem*{definition*}{Definition}

\newtheorem{remark}[thm]{Remark}

\newcommand{\CPb}{\overline{\mathbb{CP}}{}^{2}}
\newcommand{\CP}{{\mathbb{CP}}{}^{2}}
\newcommand{\CPo}{{\mathbb{CP}}{}^{1}}

\newcommand{\R}{\mathbb{R}}

\newcommand{\Z}{\mathbb{Z}}

\newcommand{\K}{{\rm K3}}

\def \x {\times}
\def \eu{{\rm{e}}}

 \def\R{{\mathbb{R}}}
 \def\Z{{\mathbb{Z}}}

\begin{document}

\title[Small Lefschetz fibrations and exotic $4$-manifolds]
{Small Lefschetz fibrations and exotic $4$-manifolds}

\author[R. \.{I}. Baykur]{R. \.{I}nan\c{c} Baykur}
\address{Department of Mathematics and Statistics, University of Massachusetts, Amherst, MA 01003-9305, USA}
\email{baykur@math.umass.edu}

\author[M. Korkmaz]{Mustafa Korkmaz}
\address{Department of Mathematics, Middle East Technical University, 06800 Ankara, Turkey}
\email{korkmaz@metu.edu.tr}

\begin{abstract}
We explicitly construct genus-$2$ Lefschetz fibrations  whose total spaces are minimal symplectic $4$-manifolds homeomorphic to complex rational surfaces $\CP \# p\, \CPb$ for  $p=7, 8, 9$, and to $3 \CP \#q\, \CPb$ for  $q =12, \ldots, 19$. \linebreak Complementarily, we prove that there are no minimal genus-$2$ Lefschetz \linebreak fibrations whose total spaces are homeomorphic to any other simply-connected $4$-manifold with $b^+ \leq 3$, with one possible exception when $b^+=3$. Meanwhile, we produce positive Dehn twist factorizations for several new genus-$2$ Lefschetz fibrations with small number of critical points, including the smallest possible example, which follow from a reverse engineering procedure we introduce for this setting. We also derive exotic minimal symplectic $4$-manifolds in the homeomorphism classes of $\CP \# 4 \CPb$ and $3 \CP \# 6 \CPb$ from small Lefschetz fibrations over surfaces of higher genera.
\end{abstract}


\maketitle

\setcounter{secnumdepth}{2}
\setcounter{section}{0}

\section{Introduction} 

Existence of minimal symplectic structures on $4$-manifolds is a fundamental question in smooth $4$-manifold topology. There has been much interest especially in producing minimal symplectic $4$-manifolds in the homeomorphism classes of standard simply-connected $4$-manifolds with small second homology, such as blow-ups of $\CP$ or $3  \,\CP$. Over the last two decades, many ingenious construction techniques, such as rational blowdowns, generalized fiber sum, knot surgery and Luttinger surgery, have been introduced to effectively address this problem by Fintushel and Stern, Gompf, Jongil Park, and several others.  (e.g.~\cite{ABP, AP, BK, Donaldson87, FSrationalblowdown, FSKnotsurgery, FSPinwheels, FSReverseEngineering, FriedmanMorgan, Gompf, Kotschick, ParkD, ParkJ, StipsiczSzabo}.) Resting on spectacular results of Taubes on Seiberg-Witten invariants of symplectic $4$-manifolds, these constructions have demonstrated that minimal symplectic $4$-manifolds not only constitute \emph{small exotic $4$-manifolds} (which are homeomorphic but not diffeomorphic to standard ones), but also resource them in almost all the known instances. 

On the other hand, since the pioneering works of Donaldson and Gompf in late $1990$s, it is known that symplectic $4$-manifolds, up to blow-ups, correspond to Lefschetz fibrations, which can be studied combinatorially as \emph{positive factorizations}, i.e. factorizations of  boundary Dehn twists into positive Dehn twists in the mapping class groups of compact oriented surfaces. This brings a priori an understanding that \emph{any} symplectic $4$-manifold, if it exists, should come from a positive factorization. In practice, however, the very idea of constructing small exotic $4$-manifolds via new positive factorizations has hardly been utilized up to date. 

The main objective of this article is to demonstrate ways to successfully implement this latter approach by constructing positive factorization for Lefschetz fibrations with small number of critical points, which we call \emph{small Lefschetz fibrations} in analogy (Section~\ref{smallLFs}), as they correspond to $4$-manifolds with small second homology.
By the virtue of our constructions, these will provide simple descriptions of many new small exotic $4$-manifolds. 

Lefschetz fibrations of fiber genus $0$ and $1$ are well-known: their total spaces are diffeomorphic to complex surfaces $\CP$, $\CPo\x\CPo$, elliptic surfaces $E(n)$, and their blow-ups \cite{Kas, Moishezon}. Therefore, for a Lefschetz fibration to bear any new symplectic $4$-manifold, the fiber genus should be at least $2$. By a thorough analysis of small genus-$2$ Lefschetz fibrations, we will show that one can already achieve a lot with genus-$2$ fibrations.

\begin{theorem} \label{mainthm}
There exist decomposable genus-$2$ Lefschetz fibrations, whose total spaces are minimal symplectic $4$-manifolds homeomorphic but not diffeomorphic to $\CP \# \,p\, \CPb$ for $p=7,8,9$, and to $3\, \CP \#\, q\, \CPb$ for $q= 12, \ldots, 19$. 
\end{theorem}

We will describe all the genus-$2$ Lefschetz fibrations in Theorem~\ref{mainthm} explicitly via positive factorizations in the mapping class group of a genus-$2$ surface with one boundary component (Section~\ref{sec:exotic}; Theorems~\ref{exoticCP} and \ref{exotic3CP}). These positive factorizations contain at least two proper (even smaller!) positive factorizations, which amounts to them decomposing as nontrivial fiber sums of Lefschetz fibrations. Based on \cite{Usher, BaykurPAMS}, this allows us to conclude that the corresponding Lefschetz fibrations are \emph{minimal} in the stronger sense: they are not only relatively minimal, i.e. not containing any exceptional spheres in the fibers, but also they do not contain any exceptional spheres at all (Proposition~\ref{minimalityprop}). We construct the exotic genus-$2$ Lefschetz fibrations of Theorem~\ref{mainthm} in two main steps. In Section~\ref{smallLFs}, we will describe a general procedure to produce small positive factorizations in the mapping class group, which we will call \emph{reverse engineering}, motivated by a similar procedure defined by Fintushel and Stern under the same name to produce small exotic $4$-manifolds \cite{FSReverseEngineering}. We employ this technique to obtain several (very) small positive factorizations, including the smallest possible positive factorization for a nontrivial genus-$2$ Lefschetz fibration (Theorem~\ref{thesmallest}). As we will observe (Proposition~\ref{difftypes}), these very small positive factorizations themselves are not exotic; they always have total spaces \emph{diffeomorphic} to complex rational and ruled surfaces. Then the second step of the construction will be to take products of certain conjugates of these positive factorizations to arrive at \emph{simply-connected} exotic Lefschetz fibrations of Theorem~\ref{mainthm}. As quick byproducts of our two step construction, we will also show that there are --mostly nonholomorphic-- genus-$2$ Lefschetz fibrations with $b^+=1$ and $c_1^2=0,1,2$, with any fundamental group $(\Z / \, m_1 \Z) \, \oplus (\Z / \, m_2 \Z)$ (Remark~\ref{noncomplex}), and also that there exist an infinite family of pairwise non-diffeomorphic minimal symplectic $4$-manifolds in the homeomorphism classes of $3\, \CP \#\, q\, \CPb$ for $q=12, \ldots, 19$ (Corollary~\ref{infinitecor}).


\smallskip
Next is a complementary result:  

\enlargethispage{0.15in}
\begin{theorem}\label{scgeography}
Any \emph{simply-connected} minimal genus-$2$ Lefschetz fibration $(X,f)$ with $b^+(X) \leq 3$ is homeomorphic to $\CP \# \, p \, \CPb$ for some $7 \leq p \leq 9$ or to $3 \, \CP \# \, q \, \CPb$ for $11 \leq q \leq  19$.
\end{theorem}

This theorem will follow from our study of the \emph{geography} of small minimal genus-$2$ Lefschetz fibrations in Section~\ref{geography}. There we  analyze which pairs $(n,s)$, where $n$ and $s$ are the number of Dehn twists along nonseparating and separating curves in a positive factorization on a genus-$2$ surface, can possibly exist. We will use a combination of obstructions coming from the algebraic topology and the Seiberg-Witten invariants of genus-$2$ Lefschetz fibrations (Lemmas~\ref{constraints} and~\ref{atmost30}) to determine the list of homeomorphism classes that can possibly support any simply-connected genus-$2$ Lefschetz fibration with $b^+ \leq 3$ (Theorem~\ref{scgeography}). The only possible case we have not been able to realize by a minimal simply-connected genus-$2$ Lefschetz fibration so far is when  $(n,s)=(10,10)$ and the total space is homeomorphic to  $3\CP \# 11 \CPb $. This is the exception noted in the theorem, which we are inclined to believe exists. 


We would like to remark that minimal, decomposable genus-$2$ Lefschetz fibrations in the homeomorphism class of $\CP \# 9 \CPb$ are already present in the literature: 
as shown by Fintushel and Stern, knot surgery on the elliptic surface $E(1) \cong \CP \# \, 9 \CPb$ with a genus-$1$ fibered knot yields a genus-$2$ Lefschetz fibration, and this fibration  decomposes into two smaller fibrations on complex ruled surfaces \cite{FSLF}, which has been seen as a very special example up until now. Also preceding our results are a few other examples of minimal simply-connected genus-$2$ Lefschetz fibrations for larger $p$ and $q$, which have been obtained through lantern substitutions corresponding to rational blowdowns: for  again $\CP \# 9 \CPb$ by Endo and Gurtas \cite{EndoGurtas} and  for $3 \CP \# q \, \CPb$ with $15 \leq q \leq 19$ by Akhmedov and Park \cite{AkhmedovParkJY, ParkJY}.

In Section~\ref{evensmaller}, we will show that Luttinger surgeries along fibered Lagrangian tori in small genus-$2$ Lefschetz fibrations over surfaces of positive genera can be used to produce minimal symplectic $4$-manifolds in even smaller homeomorphism classes (Theorem~\ref{thm:evensmaller}). Per Theorem~\ref{scgeography}, this can only be achieved by dropping the fibered aspect, and thus these examples always involve a couple of Luttinger surgeries that destroy the fibration structure. We will prove:

\begin{theorem} \label{mainthm3}
There exist decomposable minimal genus-$2$ Lefschetz fibrations over $T^2$ and $\Sigma_2$ which are equivalent via Luttinger surgeries to  minimal symplectic \linebreak $4$-manifolds homeomorphic to $\CP \# 4 \CPb$ and $3 \CP \# 6 \CPb$, respectively.
\end{theorem}

The only other examples of minimal symplectic $4$-manifolds in Theorem~\ref{mainthm3} were given by Akhmedov and Park in \cite{AP}, and later by Fintushel and Stern for $\CP \# 4 \CPb$ in \cite{FSPinwheels}. The vantage point in our constructions is that departing from Lefschetz fibrations allows us to carry on fairly simple fundamental group calculations, and the resulting exotic $4$-manifolds are easy to describe in terms of handle decompositions (Remark~\ref{experts}). 

\vspace{0.1in}
Let us finish by turning back to the fundamental point underlying our work initiated here. As discussed above, \textit{any} symplectic $4$-manifold, let it be fake, exotic, or else, would arise from a positive factorization for a Lefschetz fibration or a pencil. Our analysis of genus-$2$ Lefschetz fibrations show that minimal symplectic $4$-manifolds homeomorphic to smaller number of blow-ups of $\CP$, $\CPo \x \CPo$ or $3 \CP$ will only come from  positive factorizations in the mapping class group of higher genera surfaces. We will carry on our work in this direction elsewhere.

\vspace{0.3in}
\noindent \textit{Acknowledgements.} The main results of this article were presented at the Great Lakes Geometry Conference in Ann Arbor and the Fukaya Categories of Lefschetz Fibrations Workshop at MIT back in March 2015; we would like to thank the organizers for motivating discussions.  We also thank Hisaaki Endo and Tom Mark for their comments. The first author was partially supported by the NSF Grant DMS-$1510395$ and the Simons Foundation Grant $317732$.

\newpage
\section{Preliminaries and background results}  \label{preliminaries}

We start with a quick review of basic definitions and properties of Lefschetz fibrations and  factorizations in mapping class groups of surfaces. The reader can turn to \cite{GompfStipsicz} for more details. We will then present several preliminary results on the algebraic and differential topology of genus-$2$ Lefschetz fibrations, which we make repeated use of in the rest of the paper. 

\smallskip
\subsection{Lefschetz fibrations}  \

A \emph{Lefschetz fibration} on a closed, smooth, oriented $4$-manifold $X$ is a smooth surjective map  $f: X \to S^2$ whose critical locus consists of finitely many points $p_i$, at which and at $f(p_i)$ there are local complex coordinates (compatible with the orientations on $X$ and $S^2$) with respect to which $f$ takes the form $(z_1, z_2) \mapsto z_1 z_2$. We say $(X,f)$ is a \emph{genus-$g$ Lefschetz fibration} for $g$ the genus of a \emph{regular fiber} $F$ of $f$. Hereon we use the term Lefschetz fibration only when the set of critical points $\{p_i\}$ is nonempty, often referred as the Lefschetz fibration being \emph{nontrivial}. We moreover assume that all the points $p_i$ lie in distinct \emph{singular fibers}; this can always be achieved after a small perturbation. 

A \emph{section} of a Lefschetz fibration $(X,f)$ is an embedded $2$-sphere $S \subset X \setminus \{p_i\}$ intersecting all fibers of $f$ at one point, which, equivalently, is the image of a map $r\colon S^2 \to X$ such that $f \circ r = 1_{S^2}$. Blowing down any collection $\{S_j\}$ of  disjoint sections  of self-intersection $-1$, one obtains a \emph{Lefschetz pencil} with \emph{base points} $\{b_j\}$. It is shown by Donaldson that every symplectic $4$-manifold admits a Lefschetz pencil, whose base points can be blown up to arrive at a Lefschetz fibration~\cite{Donaldson}. Conversely, as shown by Gompf, for every nontrivial Lefschetz fibration $(X,f)$, one can construct a symplectic form on $X$, with respect to which regular fibers and any preselected collection of disjoint sections of $f$ are symplectic~\cite{GompfStipsicz}.  

A singular fiber of $(X,f)$ is called \emph{irreducible} if the complement of the critical point in the fiber is connected, and is called \emph{reducible} otherwise. Lefschetz singularities locally correspond to $2$-handle attachments to $D^2 \x F$ with framing $-1$ with respect to the fiber framing, where the attaching circles of these $2$-handles, called \emph{vanishing cycles}, are embedded curves in a regular fiber $F$. With this in mind, an irreducible singular fiber is given by a nonseparating vanishing cycle, and a reducible singular fiber is given by a separating one. When the latter is null-homotopic on $F$, one of the fiber components becomes an \emph{exceptional sphere}, an embedded $2$-sphere of self-intersection $-1$, which can be blown down without altering the rest of the fibration. 

A common way to construct new Lefschetz fibrations is the \emph{fiber sum} operation: Let $(X_i, f_i)$ be a genus-$g$ Lefschetz fibration with regular fiber $F_i$ for $i=1,2$. The \emph{fiber sum} \, is a genus-$g$ Lefschetz fibration $f$ on $X= (X_1, F_1)\#_{\phi}(X_2,F_2)$, obtained by removing a fibered tubular neighborhood of each $F_i$ and then identifying the resulting boundaries via complex conjugation on $S^1$ times a chosen orientation preserving-diffeomorphism $\phi: F_1 \to F_2$. We say a Lefschetz fibration $(X,f)$ is \emph{indecomposable} if it cannot be expressed as a fiber sum. 

A Lefschetz fibration $(X,f)$ is called \emph{relatively minimal} if there are no exceptional spheres contained in the fibers. In this paper, we will often say $(X,f)$ is \emph{minimal}, if there are no exceptional spheres in $X$ at all. There are nonminimal Lefschetz fibrations which are relatively minimal and do not have any sections of self-intersection $-1$ \cite{BaykurHayano, SatoGeography}. Non-minimal $(X,f)$ of positive genera are known to be indecomposable \cite{Usher, BaykurPAMS}. 

For genus-$2$ Lefschetz fibrations, which are the focus of this paper, we will use the following short-hand notation: a genus-$2$ Lefschetz fibration $(X,f)$ is said to be of \emph{type $(n,s)$} if it is relatively minimal and it has exactly $n$ nonseparating and $s$ separating vanishing cycles.

\smallskip
\subsection{Positive factorizations} \

Let $\Sigma_g^m$ denote a compact, connected, oriented surface genus $g$ with $m$ boundary components, and $\Gamma_g^m$ denote its \emph{mapping class group}, the group composed of orientation-preserving self homeomorphisms of $\Sigma_g^m$ which restrict to the identity along $\partial \Sigma_g^m$, modulo isotopies that restrict to the identity along $\partial \Sigma_g^m$ as well. We write $\Sigma_g = \Sigma_g^0$, and $\Gamma_g=\Gamma_g^0$ for simplicity.

We denote by $t_c \in \Sigma_g^m$ the positive (right-handed) Dehn twist along the simple closed curve $c \subset \Sigma_g^m$. Let $\{c_i\}$ be a \emph{nonempty} collection of simple closed curves on $\Sigma_g^m$, which do not become null-homotopic when $\partial \Sigma_g^m$ is capped off by disks, and let $\{\delta_j\}$ be a collection of $m$ curves parallel to distinct boundary components of $\Sigma_g^m$. For  a collection $\{ k_j \}$ of $m$ integers, if the relation 
\begin{equation} \label{factorization}
t_{c_l} \cdots t_{c_2} t_{c_1} = t_{\delta_1}^{k_1} \cdots t_{\delta_m}^{k_m} \, \, 
\end{equation}
holds in $\Gamma_g^m$, then we call the word $W$ on the left-hand side a \emph{positive factorization of length $l$} of the mapping class \, $t_{\delta_1}^{k_1} \cdots t_{\delta_m}^{k_m} $ in $\Gamma_g^n$. Capping off $\partial \Sigma_g^m$ induces a homomorphism  $\Gamma_g^m \to \Gamma_g$, under which $W$ maps to a similar positive factorization of length $l$ of the identity element $1 \in \Gamma_g$. 

The positive factorization in (\ref{factorization}) above gives rise to a genus-$g$ Lefschetz fibration $(X,f)$ with $l$ critical points and $m$ disjoint sections $S_j$ of self-intersection $S_j^2= -k_j$ \, \cite{BaykurKorkmazMonden}. Identifying the regular fiber $F$ with $\Sigma_g$, we can view the vanishing cycles of $f$ as $c_i$. In particular, we get all the information on the topology of reducible fibers (how many there are and the genera of the fiber components they split) and observe that $(X,f)$ is relatively minimal. 

In fact, Lefschetz fibrations can always be described in terms of positive factorizations. The local monodromy around the singular fiber with vanishing cycle $c_i$ is $t_{c_i}$, and thus, the global monodromy of the fibration around a disk containing all $\{f(c_i)\}$ is a product of positive Dehn twists, called a \emph{monodromy factorization}. The fact that the map extends over the base $S^2$ dictates that we get a positive factorization of the identity in $\Gamma_g$. In the presence of sections $S_j$, one can obtain a further lift of this factorization to $\Gamma_g^m$, which yields a positive factorization $W$ as in (\ref{factorization}) above  \cite{GompfStipsicz, Matsumoto, BaykurKorkmazMonden}

When $g \geq 2$, there is indeed a one-to-one equivalence between genus-$g$ Lefschetz fibrations up to isomorphisms (orientation-preserving self-diffeomorphisms of the $4$-manifold and $S^2$, which make the fibrations commute) and positive factorizations of the identity element in $\Gamma_g$ up to \textit{Hurwitz moves}  (trading subwords $t_{c_i}t_{c_{i+1}}$ with $t_{c_{i+1}}t_{c_{i+1}}^{-1}t_{c_i}t_{c_{i+1}}= t_{c_{i+1}}t_{t_{c_{i+1}}(c_i)}$) and \textit{global conjugations} (trading every $t_{c_i}$ with $t_{\phi (c_i)}$, for some $\phi \in \Gamma_g$).

\newpage
\subsection{Algebraic and differential topology of Lefschetz fibrations} \

When a Lefschetz fibration $(X,f)$ is built from a positive factorization \linebreak $W=t_{c_l} \cdots t_{c_2} t_{c_1} $ of a product of boundary twists in $\Gamma_g^m$ as in (\ref{factorization}), one has it easy with  reading off several algebraic topological invariants of $X$ from $W$. If we assume $m \geq 1$ for the moment, i.e. when $f$ has at least one section, we have 
\[ \pi_1(X) \cong \, \pi_1(\Sigma_g) \, / \, N(c_1, \ldots, c_l) \, , \]
where $N(c_1, \ldots, c_l)$ denotes the subgroup of $\pi_1(\Sigma_g)$ generated normally by $\{c_i\}$. For $\{a_j, b_j\}$ standard generators of $\pi_1(\Sigma_g)$, we therefore get
\[ \pi_1(X) \cong \, \langle \, a_1, b_1, \ldots, a_g, b_g \, | \, [a_1, b_1] \cdots [a_g, b_g] , R_1, \ldots, R_l \, \rangle \, , \]
where each $R_i$ is a relation obtained by expressing $c_i$ (oriented arbitrarily) in  $\{a_j, b_j\}$. When $m=0$, i.e. when it is not known from the positive factorization that the fibration has a section, $\pi_1(X)$ is given by a similar presentation with possibly one additional nontrivial relation.

The Euler characteristic of $X$ is given by $\eu(X)= 4-4g+l$.  A simple formula for the signature of $X$ is also available when all $\{c_i\}$ are fixed by a hyperelliptic involution on $\Sigma_g$, in which case $(X,f)$ is called \emph{hyperelliptic}.  Since this condition always holds when $g=2$ (per $\Gamma_2$ being hyperelliptic), for the purposes of this paper, we will content ourselves with this closed formula, which  reads $\sigma(X)=-\frac{1}{5}(3n+s)$ \cite{Matsumoto, Endo}, where $n$ and $s$ are the number of nonseparating and separating curves in the collection $\{c_i\}$. 

One can also define $c_1(X)$ as the first Chern class and $\chi_h(X)$ as the holomorphic Euler characteristic for any almost complex structure compatible with say a Gompf symplectic form on $(X,f)$. We will often find it more convenient to appeal to the numerical invariants $c_1^2(X)$ and $\chi_h(X)$, even though they are determined by $\eu(X)$ and $\sigma(X)$ via $c_1^2(X) = 2\eu(X) + 3\sigma(X)$ and $\chi_h(X)= \frac{1}{4}(\eu(X)+ \sigma(X))$.

Summarizing the above for the case of genus-$2$ Lefschetz fibrations, we have:

\begin{lemma}\label{charnumbers}
If $(X,f)$ is a genus-$2$ Lefschetz fibration of type $(n,s)$, then 
\begin{itemize}
\item $ \eu(X)=n+s-4$,  \ 
\item $\sigma(X)=-\frac{1}{5}(3n+s)$,   \
\item $c_1^2(X)= \frac{1}{5}(n+7s) -8$, \
\item $\chi_h(X)= \frac{1}{10}(n+2s) -1$ . \
\end{itemize}
\end{lemma}

\smallskip

Next lemma gathers a collection of necessary conditions on $n$ and $s$ for $(X,f)$ to be a genus-$2$ Lefschetz fibration of type $(n, s)$, listed in an increasing order of complexity in terms of involved arguments. These all together will provide very powerful constraints on the existence of genus-$2$ Lefschetz fibrations. 

\begin{lemma}\label{constraints}
If $(X,f)$ is a genus-$2$ Lefschetz fibration of type $(n,s)$, then
\begin{itemize}
\item $n+2s \equiv 0 \, \, (\text{\rm mod } 10)$, \
\item $2n -s \geq 3$, 
\item $n+7s \geq 20$  . \
\end{itemize}
\end{lemma}


\begin{proof}
The first relation is well-known. The first homology group of $\Gamma_2$ is $\Z_{10}$, where any Dehn twist along a nonseparating curve corresponds to $\bar{1}$ and any Dehn twist along a nontrivial separating curve corresponds to $\bar{2}$. From  the monodromy factorization of $(X,f)$ we then deduce that $n+2s$ is a multiple of $10$.

The second relation will follow from a lengthy play with the characteristic numbers of a genus-$2$ Lefschetz fibration . Picking an irreducible fiber component $F_i$ of self-intersection $-1$ from each reducible singular fiber, we obtain a disjoint collection of $s$ embedded tori. The homology classes $A_i=[F_i]$, for $i=1, \ldots, s$, generate an $s$-dimensional subspace $V$ of $H_2(X; \R)$, to which the restriction of the intersection form $Q_X$ is negative-definite. On the other hand, the first Chern class of an almost complex structure associated to $f$ evaluates on a regular fiber $F$ as \linebreak $\eu(F)=2-2g(F)= -2 \neq 0$, so $[F]\neq 0$ in  $H_2(X; \R)$. Since $[F] \in V^{\perp}$, the orthogonal complement of $V$ with respect to $Q_X$, and since $Q_X|_{V^{\perp}}$ is nondegenerate, there exists a class $A$ in 
 $V^{\perp}$ with negative square. Thus we get an $(s+1)$-dimensional subspace $W$ of $H_2(X; \R)$ generated by $A$ and all $A_i$, on which $Q_X$ is negative-definite. We conclude that 
\begin{equation}\label{b^-}
b^-(X) \geq s+1. 
\end{equation}

Now from the equalities
\[ 2 - 2b_1(X) + b^+(X) +b^-(X) = \eu(X) = n+s-4 \, \]
\[  b^+(X) - b^-(X) = \sigma(X)=-\frac{1}{5}(3n+s) \]
we derive 
\begin{equation} \label{b1}
 b^-(X) = \frac{1}{5}(4n+3s) - 3 +b_1(X) \, . 
\end{equation}
Combining it with $b^-(X) \geq s+1$ above, we get
\begin{equation} \label{b1ineq}
b_1(X) \geq 4-\frac{2}{5}(2n-s) \, . 
\end{equation}
Since $(X,f)$ should have at least one nonseparating vanishing cycle \cite{SmithHodge, StipsiczChern}, we have $b_1(X) \leq b_1(\Sigma_2) -1 =3$. So from $(\ref{b1ineq})$ we get: 
$4n -2s \geq 5 $. Since $4n -2s$ is even, we have $4n -2s \geq 6$, or equivalently $2n -s \geq 3 $. 

Lastly, if we take the fiber sum of two copies of $(X,f)$, then the resulting \linebreak genus-$2$ Lefschetz fibration $(DX, Df):= (X, f) \#_{1} (X,f)$, \textit{the double of $(X,f)$}, is minimal. It is  observed by Stipsicz \cite{StipsiczChern} that  $b^+(DX)>1$ , so by the work of Taubes \cite{Taubes, Taubes2}, $c_1^2(DX) \geq 0$. We easily calculate  $\eu(DX)= 2\eu(X) - 2 \eu(F)=2 \eu(X) +4$ and $\sigma(DX)=2\sigma(X)$ (by Novikov additivity)  to conclude that $c_1^2(DX)= 2 c_1^2(X)+8$.   In turn, $c_1^2(X) \geq -4$, which implies $n+7s \geq 20$. 
\end{proof}

\vspace{0.1in}

Finally, we consider a special class of positive factorizations we will work with. Let $W$ be a positive factorization of the form\, $W= W' W''$ in $\Gamma_g^m$, where $W', W''$ are products of positive Dehn twists along curves which do not become null-homotopic when $\partial \Sigma_g^m$ is capped off. If the product $W' = \prod t_{c_i}$, as a mapping class, commutes with  some fixed $\phi \in \Gamma_g^m$, we can produce a new positive factorization $\widetilde{W}= (W')^{\phi} W''$, where $(W')^{\phi}$ denotes the conjugate factorization
\[ (W')^{\phi}:=\phi W' \phi^{-1} = \phi (\prod t_{c_i}) \phi^{-1} = \prod (\phi t_{c_i} \phi^{-1}) = \prod  t_{\phi(c_i)} \, . \]
In particular, if $W'$ (and thus $W''$) is a positive factorization, then for \emph{any} $\phi$, we get a new positive factorization $\widetilde{W}= (W')^{\phi} W''$, since boundary parallel Dehn twists are central in $\Gamma_g^m$. We claim that in this case, a Lefschetz fibration $(\widetilde{X}, \widetilde{f})$ constructed from $\widetilde{W}$ is always minimal. Our assumption on the Dehn twist curves in $W', W''$ implies that $g \geq 1$. It follows that $(\widetilde{X},\widetilde{f})$ is a fiber sum of two nontrivial genus $g \geq 1$ Lefschetz fibrations constructed from the positive factorizations $W', W''$ \cite{BaykurKorkmazMonden}. By a theorem of Usher on \emph{symplectic sums} \cite{Usher} (or see \cite{BaykurPAMS}[Theorem~1] for a simpler proof set only for fiber sums of Lefschetz fibrations), $\widetilde{X}$ is minimal. Hence we have the following conclusion, which will become our main criterion when arguing the minimality of certain Lefschetz fibrations we will construct in this paper:

\begin{proposition} \label{minimalityprop}
Let $(\widetilde{X},\widetilde{f})$ be a Lefschetz fibration constructed from a positive factorization $\widetilde{W}= (W')^{\phi} \, W''$ in $\Gamma_g$, where $W', W''$ are positive factorizations themselves, and $\phi$ any mapping class in $\Gamma_g$. Then $\widetilde{X}$ is a minimal symplectic $4$-manifold.  
\end{proposition} 

We note that when $\phi$ is a Dehn twist $t_{\alpha}^{\pm 1}$, the symplectic Lefschetz fibration $(\widetilde{X},\widetilde{f})$ prescribed by $\widetilde{W}= (W')^{\phi} \, W''$ is obtained from the symplectic Lefschetz fibration $(X,f)$ prescribed by $W= W' W''$ via \emph{fibered Luttinger surgery} \cite{Auroux, BaykurLuttingerLF}. This is a special instance of \emph{Luttinger surgery}, where one cuts out a Weinstein tubular neighborhood of an embedded Lagrangian torus \,$T$ in a symplectic $4$-manifold $X$ and glues it back in differently in a way that the symplectic form in the complement extends \cite{Luttinger, ADK}. The new symplectic $4$-manifold $\widetilde{X}$ is determined by $T \subset X$, along with a primitive curve $\lambda$ on $T$ and an integer $k$, called the \emph{surgery curve} and the \emph{surgery coefficient}. 
In the case of fibered Luttinger surgery above, the \emph{fibered Lagrangian torus} \,$T$ in $(X,f)$ is obtained by taking the parallel transport of $\alpha$ over a loop $\gamma$ enclosing all the Lefschetz critical values coming from $W'$, and one chooses the surgery curve as $\alpha$ and the surgery coefficient $k=\pm 1$ to arrive at $(\widetilde{X},\widetilde{f})$.

\section{Small positive factorizations and Lefschetz fibrations} \label{smallLFs} 
Similar to measuring the size of a $4$-manifold $X$ by the rank of $H_2(X)$, we will measure the size of a genus-$g$ Lefschetz fibration $(X,f)$ by what amounts to $H_2(X)$ once the genus $g$ is fixed: by the number of critical points $l$.  For which values of $g,l$ a genus-$g$ Lefschetz fibration with $l$ critical points would qualify to be called ``\emph{small}'' is a question contingent to the formulation of the focus problem (as in the case of ``small'' $4$-manifolds). In this paper, we will call a genus-$2$ Lefschetz fibration small when $l \leq 30$. As we will observe in Section~\ref{geography}, for a simply-connected $X$, this is  the case for a minimal $(X,f)$ precisely when $X$ is in the homeomorphism classes of $\CP \# p \CPb$  with $p \leq 9$, or $3 \CP \# q \CPb$ with $q \leq 19$. 

This section is devoted to the study of (very) small indecomposable genus-$2$ Lefschetz fibrations.

\subsection{Reverse engineering small positive factorizations} \label{sec:3.1} \

Here we outline a general procedure to produce positive factorizations in $\Gamma_g^m$ with small numbers of Dehn twists. We will call this procedure \emph{reverse engineering}, not only because it reflects well the very nature of this process as we will explain shortly, but also to draw a parallel to a similar procedure introduced by Fintushel and Stern under the same name that can be employed effectively to construct small exotic $4$-manifolds \cite{FSReverseEngineering}.

Two invariants one can attach to a positive factorization $W$ in $\Gamma_g^m$ are: $\ell(W)$, the length of $W$, i.e. the number of positive Dehn twists in $W$,\, and $\sigma(W)$, the signature of $W$, which is defined for any positive factorization via Meyer's signature cocycle \cite{Meyer, Endo, EndoNagami}. The Euler characteristic and the signature of the corresponding Lefschetz fibration $(X,f)$ are determined by these two invariants via the  identities $\eu(X)= 4-4g+\ell(W)$ and $\sigma(X)=\sigma(W)$, and vice versa. Moreover, since \linebreak  $\eu(X)=2 -2 b_1(X) + b^+(X)+b^-(X)$, where $b_1(X) \leq 2g$, we see that a small positive factorization  corresponds to a genus-$g$ Lefschetz fibration $(X,f)$ with small $b^+(X)$ and  $b^-(X)$. Now if we would like to have a genus-$g$ Lefschetz fibration $(X,f)$ with small $b^+(X)$ (say $b^+(X)=1$, the smallest it can be), we observe that the best way to realize $b^-(X)$ with the least number of singular fibers is when we have many reducible fibers: for each reducible fiber of $(X,f)$ counts into $b^-(X)$ (cf. the proof of Lemma~\ref{constraints}, and also see Ozbagci's calculation of the local signature contributions \cite{Ozbagci}). Note that the a posteriori guidance in this case becomes a fact when $g=2$. We will therefore aim for small positive factorizations with \emph{many} separating Dehn twists. 

We will need two well-known relations in the mapping class group $\Gamma_g^m$ of $\Sigma_g^m$. Recall that if two curves are disjoint then the corresponding Dehn twists commute, and if two simple closed curves $a$ and $b$ on $\Sigma_g^m$ intersect transversely at one point, then we have the \emph{braid relation}
\[ t_a t_b t_a= t_b t_a t_b \, .\]
Note that modifying a given positive factorization by braid relations amount to applying Hurwitz moves, which do not change the isomorphism class of the corresponding Lefschetz fibration. Secondly, let $c_1,c_2,\ldots, c_{2k}$ be a \emph{chain} of simple closed curves on $\Sigma_g^m$
such that $c_i$ and  $c_j$ are disjoint if $|i-j|\geq 2$ and that $c_i$ and $c_{i+1}$ intersect at one point. Then a regular neighborhood of $c_1\cup c_2\cup \cdots \cup  c_{2k}$ is a subsurface of $\Sigma_g^m$ with one boundary component, $\delta$. We then have the \emph{$2k$-chain} relation \cite{BirmanHilden}
\[
\left( t_{c_1}t_{c_2}\cdots t_{c_{2k}} \right)^{4k+2}=t_{\delta}. 
\]
A notable consequence of this relation is that for any curve $\delta$ on $\Sigma_g^m$ bounding a subsurface  $\Sigma_{k}^1$, the separating Dehn twist $t_{\delta}$ can be replaced by a product of \linebreak $8k^2+4k$  nonseparating Dehn twists.  In turn, any positive factorization in $\Gamma_g^m$, $m\leq 1$,  can be turned into one with only nonseparating Dehn twists.

If \,$W_0$ is a positive factorization obtained from $W$ by replacing $t_{\delta}$ with the product $\left( t_{c_1}t_{c_2}\cdots t_{c_{2k}} \right)^{4k+2}$, then we have \cite{EndoNagami}[Proposition~3.9]: 
\begin{equation} \label{REequations}
\ell(W_0)=\ell(W) +8k^2+4k-1, \text{ and } \sigma(W_0)= \sigma(W)-4k^2-4k+1.
\end{equation}

We can now describe our general construction scheme, which relies on reversing the above trick.
To reverse engineer a small positive factorization of given length $\ell$ and signature $\sigma$, we first look for a \emph{model} positive factorization $W_0$, a known factorization, whose length $\ell(W_0)$ and signature $\sigma(W_0)$ can be derived from $\ell$ and $\sigma$ via $(\ref{REequations})$ by replacing a number of separating vanishing cycles. Moreover we want $W_0$ to contain many Dehn twists along chains of curves. We then apply braid relations to $W_0$, and also to any other word we derive at intermediate steps, to get factors conjugate to $\left( t_{c_1}t_{c_2}\cdots t_{c_{2k}} \right)^{4k+2}$, for some $1 \leq k \leq g-1$. These factors will be then traded with separating Dehn twists to arrive at a positive factorization $W$ with desired $\ell$ and $\sigma$.

\vspace{0.1in}
\noindent \textit{When $g=2$: } Our procedure is much simplified and easier to demonstrate in this case. Suppose that $m\leq 1$. There is only one type of nontrivial separating vanishing cycle other than curve parallel to boundary, and it splits off genus-$1$ subsurfaces on both sides, allowing only a $2$-chain substitution: 
\[ ( t_{c_1}t_{c_2})^{6}=t_{\delta}. \]
Moreover, when $g=2$, $\ell$ and $\sigma$ determine the number of nonseparating and separating vanishing cycles, $n$ and $s$. Call the positive factorization \emph{of type $(n,s)$} in this case. To produce a small positive factorization of type $(n,s)$ in $\Gamma_2^m$, we look for a model positive factorization with only $n+12s$ nonseparating vanishing cycles, and then employ braid relations to get subwords conjugate to $(t_{c_1}t_{c_2})^{6}$. Note, moreover, that we can employ Lemma~\ref{constraints} first to check whether a positive factorization with such $n, s$ is even plausible (see Section~\ref{geography}).

\vspace{0.1in}
\noindent \textit{An easy example:} Let us demonstrate how quickly one can reverse engineer  positive factorizations of types $(18,1)$ and $(6,2)$ in $\Gamma_2^2$. Here, we need a model factorization of type $(30,0)$, and there are not so many of them! The well-known analogue of the chain relation for a $5$-chain $c_1, \ldots, c_5$ yields the positive factorization in $\Gamma_2^2$ 
\begin{equation}\label{5chain}
 t_{\delta_1} t_{\delta_2} = (t_1 t_2 t_3 t_4 t_5)^6 \, ,
\end{equation}
which will be our model factorization $W_0$. Here we have simplified our notation by letting $t_i$ denote the Dehn twist $t_{c_i}$ along $c_i$. (We may assume that the curves $c_i$ are given in Figure~1, where one of the boundary components of $\Sigma_2^2$ is capped off.) Since $c_1, c_2$ are disjoint from $c_4, c_5$, the factors $t_1 t_2$ and $t_4 t_5$ commute. We may rewrite $W_0$ as
\[ 
t_{\delta_1} t_{\delta_2}= P_1 (t_1 t_2)^6 P_2 (t_4 t_5)^6   \, , 
\]
where $P_i$ is a product of $3$ positive Dehn twists that are all conjugates of $t_3$. By the $2$-chain relation, we get new positive factorizations  $P_1 t_c  P_2 (t_4 t_5)^6 $ (or $P_1(t_1 t_2)^6 P_2 t_{c'}) $ and $P_1 t_cP_2 t_{c'} $ of $t_{\delta_1} t_{\delta_2} $ in $\Gamma_2^2$, which are of types $(18,1)$ and $(6,2)$, respectively. Here $c$ (resp. $c'$) is the boundary component of a regular neighborhood of $c_1\cup c_2$ (resp. $c_4\cup c_5$). 
It can shown that the conjugation of the last factorization by the inverse of $t_4t_5t_4$ is equal to a lift of the well-known positive factorization of the identity obtained by Matsumoto we will discuss shortly.\footnote{The generalization of Matsumoto's fibration to $g>2$ in \cite{Korkmaz} can also be obtained by the same method; this will appear in \cite{DMP}.}

\smallskip
\subsection{The smallest genus-$2$ Lefschetz fibration} \

\begin{figure}[p!] 
   \begin{center}
   \includegraphics[width=10.8cm]{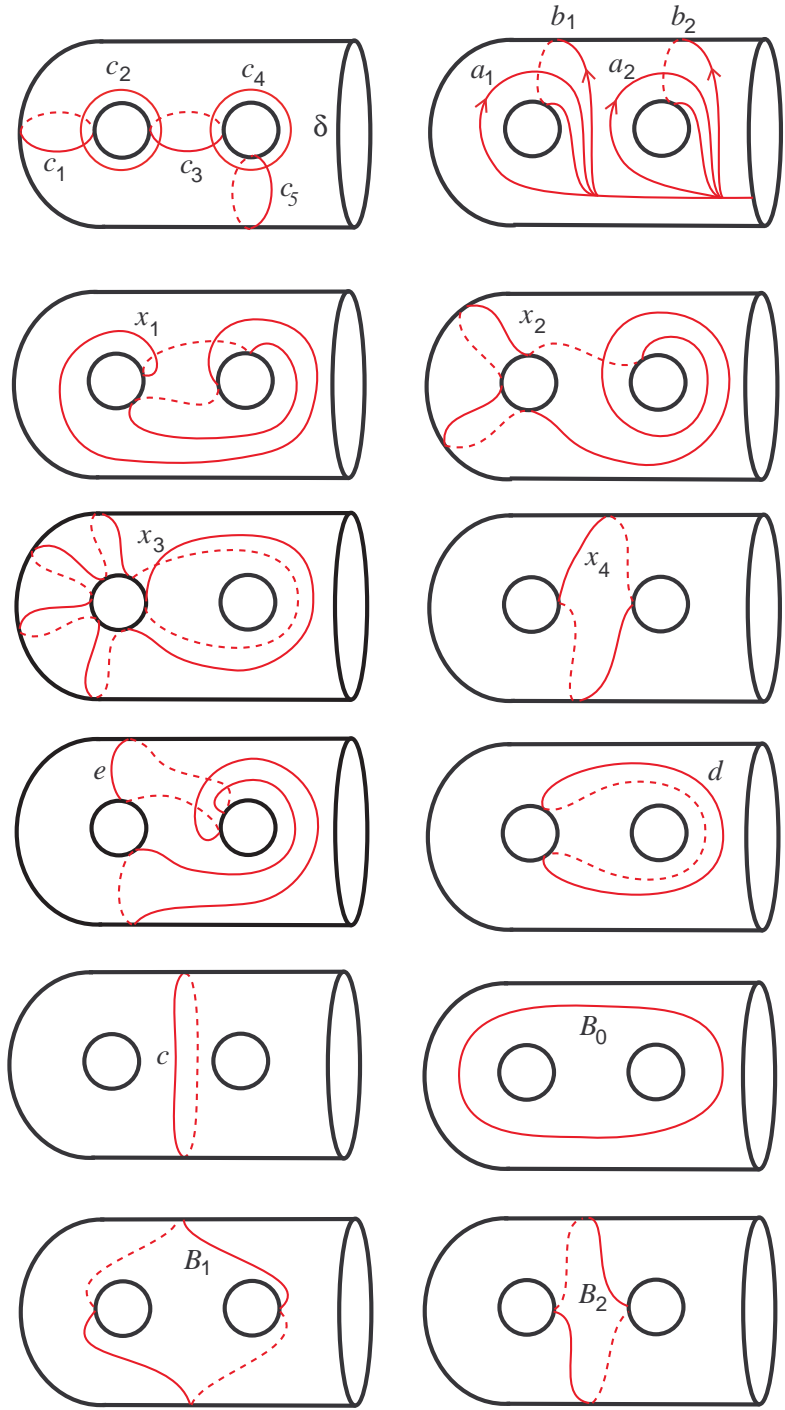}
    \end{center}
   \smallskip
   \caption{Dehn twist curves $\delta$, $c_i$, $x_i$, $B_i$, $c$, $d$, $e$ on $\Sigma_2^1$, and $\pi_1(\Sigma_2)$ generators $a_i, b_i$.}
   \end{figure} \label{curves}

 Our aim in this section is to derive the smallest genus-$2$ Lefschetz fibration. This is equivalent to find a factorization of the identity in the mapping class group $\Gamma_2$
of the smallest length. Matsomoto's well-known relation, also obtained above, is of type  $(n,s)=(6,2)$. So the length of the smallest factorization is less than or equal to $8$.
By Lemma~\ref{constraints}, there is no factorization of length less than $7$, and if there is a factorization of length $7$ then $(n,s)=(4,3)$. 

Suppose that we have a factorization of the identity of type $(4,3)$. By trading three separating Dehn twists by $12$ nonseparating Dehn twists, we get a factorization of
of type $(40,0)$. We know a factorization of this type is
\begin{equation}
(t_{c_1}t_{c_2}t_{c_3}t_{c_4})^{10}=1, \label{eqn:L40}
\end{equation}
where $c_i$ are simple the curves in Figure~1; the closed surface $\Sigma_2$ is $\Sigma_2^1$ together with a disk glued along the boundary. This is our starting point: we will 
go backward starting from (\ref{eqn:L40}), so that the positive factorization (\ref{eqn:L40}) will be our model for reverse engineering a positive factorization of  type $(4,3)$.
Instead of doing this computation in $\Sigma_2$ we will do it in $\Sigma_2^1$ instead, so that the Lefschetz fibration corresponding
to the factorization will have a section of square $-1$. 

In the mapping class group $\Gamma_2^1$, consider the $4$-chain relation 
\begin{eqnarray} \label{eqn:4-chain}
 (t_1 t_2 t_3 t_4)^{10} = t_{\delta}, \,  
\end{eqnarray}
which turns into (\ref{eqn:L40}) when a disk is glued to $\Sigma_2^1$.
From the braid relation, we see that 
\[
t_i  \left( t_1 t_2 t_3 t_4 \right)= \left( t_1 t_2 t_3 t_4 \right) t_{i-1}\]
 for $i=2,3,4$. Using this we obtain
	\begin{eqnarray*}
	\left( t_1 t_2 t_3 t_4 \right)^5 
	    &=& \left( t_1 t_2 t_3 t_4 \right)\left( t_1 t_2 t_3 t_4 \right)\left( t_1 t_2 t_3 t_4 \right)\left( t_1 t_2 t_3 t_4 \right)\left( t_1 t_2 t_3 t_4 \right)\\
	    &=& \left( t_1 t_2 t_3  \right)\left( t_1 t_2 t_3  \right)\left( t_1 t_2 t_3  \right) \left( t_1 t_2 t_3\right) \left(t_4  t_3t_2t_1\right)\left( t_1 t_2 t_3 t_4 \right)\\
	    &=& \left( t_1 t_2 \right)\left( t_1 t_2  \right)\left( t_1 t_2 \right)     \left( t_3 t_2 t_1 \right) \left( t_1 t_2 t_3\right) \left(t_4  t_3t_2t_1\right)\left( t_1 t_2 t_3 t_4 \right)\\
	    &=& \left(t_1t_2 \right)^3 \left( t_3t_2t_1t_1t_2t_3 \right)  \left(  t_4 t_3t_2t_1t_1t_2t_3 t_4\right).
	\end{eqnarray*}
We note that the factors $(t_1t_2)^3$, $( t_3t_2t_1t_1t_2t_3 )$, and $( t_4 t_3t_2t_1t_1t_2t_3 t_4)$ on the right all commute with each other: This can be seen from the
fact that $( t_4 t_3t_2t_1t_1t_2t_3 t_4)$ fixes the curves $c_1,c_2$ and $c_3$, and $( t_3t_2t_1t_1t_2t_3 )$ fixes the curves $c_1$ and $c_2$.

Starting with the $4$-chain relation (\ref{eqn:4-chain}) and then applying braid relations and the fact that Dehn twists about disjoint curves commute in $\Gamma_2^1$, we get 
\begin{eqnarray*}
t_\delta 
&=&\left( t_1 t_2 t_3 t_4 \right)^{10 }\\
&=& \left(t_1t_2 \right)^6 
       \left( t_3t_2t_1t_1t_2t_3 \right) \left( t_4 t_3t_2t_1t_1t_2t_3 t_4 \right)  
        \left( t_4  t_3t_2t_1t_1t_2t_3t_4 \right) \left(  t_3t_2t_1t_1t_2t_3 \right)\\
&=& t_c   t_3   t_2t_1t_1t_2  \left(  t_4 t_3t_2t_1t_1t_2t_3 t_4 \right)  t_3 t_3   
         \left( t_4  t_3t_2t_1t_1t_2t_3  t_4\right)   t_2t_1t_1t_2t_3 \\
&=& t_c   t_3(t_2t_1t_1t_2) t_4 t_3(t_2t_1t_1t_2)  (t_3 t_4)^3 
        ( t_2t_1t_1t_2)t_3  (t_2t_1t_1t_2)t_4t_3. 
\end{eqnarray*}
Keeping in mind that $t_\delta$ is in the center of $\Gamma_2^1$, we conjugate both sides of this equality by $t_4t_3$ and get
\begin{eqnarray*}
t_\delta 
&=& t_4t_3t_c   t_3(t_2t_1t_1t_2) t_4 t_3(t_2t_1t_1t_2)  (t_3 t_4)^3 
        ( t_2t_1t_1t_2)t_3  (t_2t_1t_1t_2) \\
&=& t_ e     t_4t_3t_3  t_4 (t_2t_1t_1t_2) t_3(t_2t_1t_1t_2)  (t_3 t_4)^3 
        ( t_2t_1t_1t_2)(t_2t_1t_1t_2) t_{x_4 },
\end{eqnarray*}
where $e=  t_4t_3(c)$ and $x_4=  (t_2t_1t_1t_2)^{-1}  (c_3)$. Here, we also used the fact that $t_4$ commutes with $t_1$ and $t_2$.

The use of the braid relations   and $ft_af^{-1}=t_{f(a)}$ give
\begin{eqnarray*}
t_\delta 
&=& t_ e     (t_4t_3t_3  t_4) t_2t_1t_1t_3 t_2t_3t_1t_1t_2)  (t_3 t_4)^3 
        ( t_2t_1t_1t_2t_2t_1t_1t_2) t_{ x_4 }   \\
&=& t_ e     t_{x_1} (t_4t_3t_3  t_4 t_1t_1t_3) t_2t_3t_1t_1t_2  (t_3 t_4)^3 
        ( t_2t_1t_1t_2t_2t_1t_1t_2) t_{ x_4 }   \\
&=& t_ e     t_{x_1}t_{x_2} (t_4t_3t_3  t_4 t_1t_1t_3t_3t_1t_1)t_2  (t_3 t_4)^3 
        ( t_2t_1t_1t_2t_2t_1t_1t_2) t_{ x_4 }   \\
&=& t_ e     t_{x_1}  t_{x_2} t_{x_3}  (t_4t_3t_3  t_4 t_1t_1t_3 t_3t_1t_1 )   (t_3 t_4)^3 
        ( t_2t_1t_1t_2t_2t_1t_1t_2) t_{ x_4 },
\end{eqnarray*}
where  $x_1=  t_4t_3t_3  t_4 (c_2)$, $x_2=  t_1t_1t_3  (x_1)$ and $x_3=  t_1t_1t_3  (x_2)$.
This may be rewritten as 
\begin{eqnarray*}
&=& t_ e     t_{x_1}  t_{x_2} t_{x_3}  t_4t_3t_3  t_4 t_3 t_3    (t_3 t_4)^3 
        (t_1 t_1t_1t_1t_2t_1t_1t_2 t_2t_1t_1t_2) t_{ x_4 }   \\
&=& t_ e     t_{x_1}  t_{x_2} t_{x_3}   (t_3 t_4)^6 (t_1t_2)^6  t_{ x_4 }.
\end{eqnarray*}
Finally, by the $2$-chain relation, we write $t_d=  (t_3t_4)^6$  and $t_c=  (t_1t_2)^6$ to arrive at the desired result
\begin{eqnarray*}
t_\delta &=& t_ e     t_{x_1}  t_{x_2} t_{x_3}   t_d  t_c  t_{ x_4 }.
\end{eqnarray*}
The nonseparating curves $x_i$, and the separating curves $c,d,e$ are as shown in Figure~1. 

Set $\underline{W_1= t_ e     t_{x_1}  t_{x_2} t_{x_3}   t_d  t_c  t_{ x_4 } }$, so it is a positive factorization of $t_{\delta}$ in $\Gamma_2^1$ and let $(X_1, f_1)$ be the corresponding Lefschetz fibration of type $(4,3)$, with a section of self-intersection $-1$. Since any genus-$2$ Lefschetz fibration should have at least $7$ singular fibers \cite{Ozbagci, SatoGeography}, which is also deduced from Lemma~\ref{constraints},  $(X_1, f_1)$ realizes the smallest possible genus-$2$ Lefschetz fibration:

\begin{theorem} \label{thesmallest}
The positive factorization $W_1=t_ e     t_{x_1}  t_{x_2} t_{x_3}   t_d  t_c  t_{ x_4 } = t_\delta$ in $\Gamma_2^1$ prescribes \emph{the smallest} genus-$2$ Lefschetz fibration. 
\end{theorem}

We note that we do not claim that the smallest genus-$2$ Lefschetz fibration is \textit{unique}  up to isomorphism. However, it is true that the total space is unique up to diffeomorphism. As seen above, a smallest genus-$2$ Lefschetz fibration is of type $(4,3)$.  By Proposition~\ref{difftypes} given at the end of this section, any such Lefschetz fibration has total space diffeomorphic to $(S^2 \x T^2) \, \# \, 3 \CPb$. For instance, the semi-stable holomorphic genus-$2$ fibration with $7$ critical points constructed by Xiao in \cite{Xiao}, even though it lacks any further information on its monodromy or the total space, would be supported on the same manifold.

\smallskip
We finish by giving an explicit presentation of $\pi_1(X_1)$ induced by $W_1$. 
In a group $G$, let $\bar{a}$ denote the inverse of $a$ and let $[a,b]$ denote the commutator $ab\bar{a}\bar{b}$. 

Since the Lefschetz fibration $(X_1, f_1)$ has a section,  we get a presentation of $\pi_1(X_1)$, from the monodromy,
with generators $a_1, b_1, a_2, b_2$ as in Figure~1, and with the following defining relators:
\begin{eqnarray} 
& \phantom{o} &  [a_1,b_1][a_2,b_2]=1,  \label{eqn:prestn0}  \\
& \phantom{o} &  a_1 r^2=1,                    \label{eqn:prestn1}		\\  
& \phantom{o} &  a_1  \bar{b}^2_1 r a_2=1,			\label{eqn:prestn2}		\\
& \phantom{o} &  a_1  \bar{b}^4_1 r   a_2   \bar{r} a_2 =1,	\label{eqn:prestn3}	\\
& \phantom{o} &  a_1  b_1 \bar{a}_1    a_2  b_2 \bar{a}_2 =1,	\label{eqn:prestn4}	\\
& \phantom{o} &  a_1  b_1 \bar{a}_1     \bar{b}_2   \bar{a}_2   r  =1,\label{eqn:prestn5}	\\
& \phantom{o} &  [r,  \bar{a}_2]   = 1,  \label{eqn:prestn6}\\
& \phantom{o} &  [a_1 , b_1]   = 1. \label{eqn:prestn7}
\end{eqnarray}
Here $r= \bar{b}_1a_2b_2$. The relators~(\ref{eqn:prestn1}),(\ref{eqn:prestn2}),\ldots, (\ref{eqn:prestn7}) come from the 
vanishing cycles $x_1,x_2,x_3,x_4,e,d,c$ respectively.

The relations~(\ref{eqn:prestn0}) and~(\ref{eqn:prestn7}) imply that $[a_2,b_2]=1$. We may then rewrite the relation~(\ref{eqn:prestn4}) as 
\[ b_1 b_2 = (a_1  b_1 \bar{a}_1)  (a_2  b_2 \bar{a}_2) = 1 . \]
With this new relation, all the other relations either trivialize or become 
\[ a_1 a_2^2 b_2^4 =1 , \]
which shows that $a_1$ can be generated by $a_2, b_2$. Thus $\pi_1(X)$ is a free abelian group of rank $2$ generated by  $a_2$ and $b_2$
(also by $a_2$ and $b_1$).

\smallskip
\subsection{Other small genus-$2$ Lefschetz fibrations} \

The positive factorization $W$ in $\Gamma_2^2$ of type $(6,2)$ we reverse engineered as a sampler at the end of Section~\ref{sec:3.1}
is a lift of  a well-known positive factorization constructed by Matsumoto \cite{Matsumoto}, which we will favor over $W$ because of its neatly symmetric presentation. For the simple closed curves $B_0,B_1,B_2,c$   on $\Sigma_2^1$ given in Figure~1, we also have \cite{Matsumoto, Korkmaz}:
\[ \left( t_{B_0} t_{B_1} t_{B_2} t_{c}\right)^2= t_{\delta} \]
in  $\Gamma_2^1$. The positive factorization $\underline{ W_2=(t_{B_0} t_{B_1} t_{B_2} t_{c} )^2}$ of $t_{\delta} $ then gives Matsumoto's genus-$2$ Lefschetz fibration of type $(6,2)$, with a section of self-intersection $-1$. 

The fundamental group $\pi_1(X_2)$ of $X_2$ admits a presentation with generators $a_1,b_1,a_2,b_2$ and with relations 
\begin{eqnarray}
 & \phantom{o} &   [a_1,b_1][a_2,b_2]=1, \nonumber \\
 & \phantom{o} &   a_1a_2=1, \label{eqn:prestn8} \\
 & \phantom{o} &   a_1 \bar{b}_1 a_2 \bar{b}_2  =1, \label{eqn:prestn9} \\
 & \phantom{o} &  b_2b_1  =1,  \label{eqn:prestn10} \\
 & \phantom{o} &  [a_1,b_1]  =1, \nonumber
\end{eqnarray}
which is easily seen to be a free abelian group of rank $2$ generated by $a_1$ and $b_1$ (also by $a_2$ and $b_2$). 
The last four relations come from the vanishing cycles $B_0,B_1,B_2$ and $c$, respectively.

Three more positive factorizations will be needed in our later constructions. 

The first one is the positive factorization we obtained in the process of deriving  $W_1$. That is, we can trade back $t_c$ in $W_1$ with $(t_{c_1} t_{c_2})^6$ and obtain a positive factorization $\underline{ W_3= t_ e     t_{x_1}  t_{x_2} t_{x_3}   t_d \, (t_{c_1} t_{c_2})^6  \, t_{ x_4 } }$ of  $t_\delta$ in $\Gamma_2^1$. Let $(X_3, f_3)$ be the corresponding \linebreak genus-$2$ Lefschetz fibration of type $(16,2)$. The fundamental group $\pi_1(X_3)$ has a presentation with generators $a_1, b_1, a_2, b_2$ and with relations $(\ref{eqn:prestn0})-\ref{eqn:prestn6})$ together with
the additional relations 
\begin{eqnarray}
 & \phantom{o}   b_1 =1,   \label{eqn:prestn11}  \\
 & \phantom{o}   a_1 =1,  \label{eqn:prestn12} 
\end{eqnarray}
induced by the new vanishing cycles $c_1$ and $c_2$.  An easy calculation shows that $\pi_1(X_3)$ is isomorphic to the cyclic group $\Z_2$ of order two, and is generated by $a_2$. 

The second one is derived from $W_2$ by trading one $t_c$ with $(t_{c_1} t_{c_2})^6$, so we get a  positive factorization \,$\underline{ W_4= t_{B_0} t_{B_1} t_{B_2} t_{c} t_{B_0} t_{B_1} t_{B_2} (t_{c_1} t_{c_2})^6 }$ of $t_{\delta}$ in $\Gamma_2^1$. Let $(X_4,f_4)$ be the corresponding genus-$2$ Lefschetz fibration of type $(18,1)$. As we already had $\pi_1(X_2) = \pi_1(\Sigma_2) / N (B_0, B_1, B_2, c) = \Z^2$ generated by $a_1, b_1$, adding the new relations induced by $c_1, c_2$ above gives $\pi_1(X_4) \cong 1$.

The last positive factorization is well-known: for $c_1, \ldots, c_5$  the curves given in Figure~1, one gets a lift  to $\Gamma_2^1$  of 
a positive factorization of the hyperelliptic involution on $\Sigma_2$, which can be then squared to get
\[  
(t_{c_1} t_{c_2} t_{c_3} t_{c_4} t_{c_5} t_{c_5} t_{c_4} t_{c_3} t_{c_2} t_{c_1})^2 =t_{\delta}  \, .
\]
Taking this positive factorization $\underline{ W_5 = (t_{c_1} t_{c_2} t_{c_3} t_{c_4} t_{c_5} t_{c_5} t_{c_4} t_{c_3} t_{c_2} t_{c_1})^2 }$, we arrive at a genus-$2$ Lefschetz fibration $(X_5, f_5)$ of type $(20,0)$. Since the vanishing cycles $c_1,c_2,c_3,c_4,c_5$ represent the free homotopy classes of 
$b_1, a_1,\bar{b}_1  a_2b_2 \bar{a}_2, a_2, b_2$ respectively, 
we easily see that $\pi_1(X_5) = 1$.

\smallskip
\subsection{Diffeomorphism types of some small genus-$2$ Lefschetz fibrations}  \

We now give a  sufficient condition for the total space of a genus-$2$ Lefschetz fibration $(X,f)$ to be a rational or ruled surface, which allows one to easily identify the \textit{diffeomorphism} type of  $X$. The proof of the next lemma generalizes a circle of ideas employed by Sato in \cite{SatoGeography}. 

\begin{lemma}\label{rationalruled}
Let $(X,f)$ be a genus-$2$ Lefschetz fibration of type $(n,s)$. If $n+7s < 30$, then $X \cong (S^2 \x \Sigma_h)  \# \, k \CPb$, where $h=2- \frac{1}{10}(n+2s)$ and $k= -\frac{1}{5}(3n+s)$.
\end{lemma}

\begin{proof}
Let $(X,f)$ be a genus-$2$ Lefschetz fibration satisfying the conditions of the lemma. To give our proof by contraposition, assume that $X$ is not a rational or a ruled surface.

Let $X_0$ be a minimal model for $X$, so $X \cong X_0 \# \, m \CPb$ for some non-negative integer $m$. It is easily seen that $c_1^2(X_{0})=c_1^2(X)+m$. By Taubes \cite{Taubes, Taubes2} and Li-Liu \cite{LiLiu}, any minimal symplectic $4$-manifold which is not rational or ruled has a non-negative $c_1^2$. It follows that
\[ 0 \leq  c_1^2(X) +m = \frac{1}{5}(n+7s) -8 + m \, . \]
On the other hand, there exist $m$ disjoint exceptional spheres $S_1, \ldots, S_m$ in $X$, each one of which is a positive multisection of $(X,f)$ (see \cite{SatoGeography, BaykurHayano} or for the most explicit discussion \cite{BaykurPAMS}[Proof of Theorem~1]), and we have 
\[m \leq \sum_i S_i \cdot F \leq 2g-2 = 2 \, . \]
Combining the two inequalities we obtain $0 \leq \frac{1}{5}(n+7s) -8 + 2$,  that is, $n+7s \geq 30$. 

By contraposition, we conclude that $n + 7s < 30$ implies $X$ is rational or ruled, so it is diffeomorphic to either $\CP$, an $S^2$-bundle over a Riemann surface $\Sigma_h$, or a blow-up of these. Note that $X$ cannot be minimal: minimal rational or ruled surfaces have signature $0$ or $1$, whereas for any \textit{nontrivial} genus-$2$ Lefschetz fibration (which we always assume to be the case) we have $\sigma(X)= -\frac{1}{5}(3n+s)<0$. The same argument rules out $\CP \# \CPb$. Therefore the diffeomorphism type of $X$ can be uniquely identified as $(S^2 \x \Sigma_h)  \# \, k \CPb$ for some non-negative integers $h$ and $ k$. These integers can be easily calculated as $h=1-\chi_h(X)$ and $k= -\sigma(X)$, and from Lemma~\ref{charnumbers} we arrive at the given calculations of $h$ and $k$ in terms of $n$ and $s$.

\end{proof}

\smallskip
With the constraints from Lemma~\ref{constraints}, there are very few possibilities for a genus-$2$ Lefschetz fibration of type $(n,s)$ with $n+7s < 30$. Namely, $(4,3), (6,2)$, $(18,1)$ and $(20,0)$.  The lemma above nevertheless makes it possible to identify the diffeomorphism types of the total spaces of these small genus-$2$ Lefschetz fibrations given earlier in Section~\ref{smallLFs}: 

\begin{proposition} \label{difftypes}
Any genus-$2$ Lefschetz fibrations of types $(4,3), (6,2), (18,1)$ and $ (20,0)$ have total spaces $(S^2 \x T^2)  \, \# 3 \CPb$, $(S^2 \x T^2)  \, \# 4 \CPb$, $\CP  \, \#  12 \CPb$ and $\CP  \, \#  13 \CPb$, respectively, and they are indecomposable. In particular the Lefschetz fibrations $(X_1, f_1)$, $(X_2, f_2)$, $(X_4, f_4)$, $(X_5, f_5)$ have total spaces diffeomorphic to these ruled manifolds, in the given order. 
\end{proposition}

\smallskip
\begin{remark}
Although this proposition gives a unified answer, we would like to point out that partial results were already known: total spaces of the particular fibrations $(X_2,f_2)$ and $(X_5,f_5)$ were already known to Matsumoto, who constructed them as double branched covers of rational ruled surfaces \cite{Matsumoto}, and the case of a potential $(4,3)$ fibration was covered by Sato in \cite{SatoGeography} (also see \cite{ParkJY}). 
\end{remark}

\begin{remark} \label{(16,2)}
A genus-$2$ Lefschetz fibration of type $(16,2)$ presents a borderline case for our lemma above, for which we can stretch our arguments to  identify the \emph{homeomorphism} class of $X_3$, the total space of the particular fibration $(X_3, f_3)$ we have constructed earlier. In this case, from  $c_1^2(X_3)=-2$, we get $m \geq 2$. By \cite{SatoKodaira, BaykurHayano}, we conclude that $X_3$ is either twice blow-up of a symplectic Calabi-Yau surface, or it is a rational or a ruled surface. As $\pi_1(X_3)= \Z_2$, the latter is ruled out. As in \cite{BaykurHayano}[Proof of Proposition~4.8], one can then conclude that $X_3$ is homeomorphic to the Enriques surface blown-up twice.

\end{remark}

\section{Geography of small simply-connected genus-$2$ Lefschetz fibrations} \label{geography}

Here we will determine which simply-connected $4$-manifolds with $b^+ \leq 3$ can admit minimal genus-$2$ Lefschetz fibrations. 

\begin{lemma} \label{atmost30}
For $n+2s \leq 30$, only the following pairs $(n,s)$ can be realized as \linebreak a \emph{minimal} genus-$2$ Lefschetz fibration of type $(n,s)$:
\begin{itemize}
\item  $(6,7), (8,6), (10,5), (12,4)$, and   \
\item  $(8,11), (10,10), (12,9), (14,8), (16,7), (18,6), (20,5), (22,4), (24,3), (26,2)$.
\end{itemize}
\end{lemma}

\begin{proof}
Let $(X,f)$ be a minimal genus-$2$ Lefschetz fibration of type $(n,s)$. Since $X$ should have at least one nonseparating vanishing cycle \cite{SmithHodge}, we will assume $n>0$. 

Since it is a minimal symplectic $4$-manifold, we have $ 0 \leq c_1^2(X)= \frac{1}{5}(n+7s) - 8$ \cite{Taubes, Taubes2}, which simplifies to
 \begin{equation} \label{minimalityinequality}
n+7s \geq 40
\end{equation}
Thus we have a strengthened version of the third constraint in Lemma~\ref{constraints}. 

Our proof will follow from a routine check to see which possible pairs $(n,s)$ meet the constraints of Lemma~\ref{constraints}, along with the strengthened inequality above. 

Clearly, no $(n,s)$ can satisfy the inequality $(\ref{minimalityinequality})$ when  $n+2s=10$.

For $n+2s=20$, the pairs $(2,9), (4,8)$ are eliminated using the second inequality in Lemma~\ref{constraints}, and $(14,3), (16,2), (18,1), (20,0)$ by $(\ref{minimalityinequality})$. 

Lastly, for $n+2s=30$, the pairs $(2,14), (4,13), (6,12)$ are ruled out by the second inequality in Lemma~\ref{constraints}, and $(28,1), (30,0)$ by $(\ref{minimalityinequality})$.
\end{proof}

We are now ready to determine all possible homeomorphism types of small \textit{simply-connected} minimal genus-$2$ Lefschetz fibrations: 

\begin{theorem}\label{scgeography}
Any \emph{simply-connected} minimal genus-$2$ Lefschetz fibration $(X,f)$ with $b^+(X) \leq 3$ is homeomorphic to $\CP \# \, p \, \CPb$ for some $7 \leq p \leq 9$ or to $3 \, \CP \# \, q \, \CPb$ for $11 \leq q \leq  19$.
\end{theorem}

\begin{proof} 
Let $(X,f)$ be a simply-connected \emph{minimal} genus-$2$ Lefschetz fibration with $b^+(X) \leq 3$. Since $X$ admits an almost complex structure and $b_1(X)=0$, $b^+(X)$ should be odd (e.g. $\chi_h(X)$ should be an integer). So $b^+(X)=1$ or $3$.

From the equation~(\ref{b1}), $b^-(X) = \frac{1}{5}(4n+3s) - 3 +b_1(X)=\frac{1}{5}(4n+3s) - 3 $. Combining it with $b^+(X)-b^-(X)= \sigma(X) = -\frac{1}{5}(3n+s)$, we get 
\begin{equation}\label{b^+}
b^+(X)= \frac{1}{5}(n+2s)  - 3.
\end{equation}
So we are in the ballpark of Lemma~\ref{atmost30}, since $n+2s \leq 30$ for $b^+(X) \leq 3$.

For $b_1(X)=0$, the second constraint in Lemma~\ref{constraints} can be improved as follows: The inequality $(\ref{b1})$ now reads $b^-(X) = \frac{1}{5}(4n+3s) - 3$, so with the inequality $(\ref{b^-})$, which states $b^-(X) \geq s+1$, we arrive at $2n-s \geq 10$. This last inequality then rules out a couple more $(n,s)$ pairs allowed by Lemma~\ref{atmost30}; namely $(6,7)$ and $(8,11)$. 

We are thus left with the possible values $(8,6), (10,5), (12,4)$ when $b^+=1$ and $(10,10), (12,9), (14,8), (16,7), (18,6), (20,5), (22,4), (24,3), (26,2)$  when $b^+=3$. In all these cases $s>0$, so the presence of a reducible fiber implies that there exists a homology class with an odd intersection number. It follows that $X$ cannot be spin. As $X$ is assumed to  be simply-connected, using Freedman's seminal work, we can conclude that $X$ is homeomorphic to $\CP \# \, p \, \CPb$ or to $3 \, \CP \# \, q \, \CPb$. Calculating $b^-(X)$ for the above list of pairs $(n,s)$, we see that  $7 \leq p \leq 9$ and $11 \leq q \leq  19$. 
\end{proof}

\smallskip

When the number of singular fibers $l=n+s \leq 30$, we obviously have \linebreak $n+2s \leq 30$, so our take of ``small'' minimal genus-$2$ Lefschetz fibrations are indeed limited to the types in Lemma~\ref{atmost30}, and in the simply-connected case, to the homeomorphism classes listed in Theorem~\ref{scgeography}. Using Lemma~\ref{charnumbers}, we can determine the possible characteristic numbers $c_1^2$ and $\chi_h$ of all minimal simply-connected genus-$2$ Lefschetz fibrations with at most $30$ singular fibers, as in the \emph{geography problem} for minimal simply-connected symplectic $4$-manifolds or complex surfaces. We will realize these lattice points in the next section.

\begin{remark} 
What if we consider the geography problem for small \emph{simply-connected} genus-$2$ Lefschetz fibrations which are asked to be only relatively-minimal? We can give a fairly complete answer to this question as well. The possible pairs $(n,s)$ ruled out in the proof of Lemma~\ref{atmost30} using minimality were the followings: any $(n,s)$ with $n+2s=10$, $(14,3), (16,2), (18,1), (20,0)$, and $(28,1), (30,0)$. 

For $n+2s=10$, the pairs $(2,4)$, $(8,1)$ and $(10,0)$ are all ruled out by Lemma~\ref{constraints}. Proposition~\ref{difftypes} shows that any genus-$2$ Lefschetz fibration of type $(4,3)$ or $ (6,2)$ has a non-simply-connected total space. 

The pairs $(18,1)$, $(20,0)$ are readily realized by the fibrations $(X_4, f_4)$ and $(X_5,f_5)$ discussed in the previous section. The $5$-chain relation
\[\left( t_{c_1}t_{c_2}t_{c_3}t_{c_4} t_{c_{5}} \right)^{6}=1 \]
in $\Gamma_2$ hands a genus-$2$ Lefschetz fibration of type $(30,0)$. Applying braid relations we can rewrite the left-hand side of the above equation in the form of $t_{c_1}^2 t_{c_3}^2 P $, where $P$ is a product of $26$ positive Dehn twists. Using the lantern relation so as to trade the subword $t_{c_1}^2 t_{c_3}^2$ with a product of three Dehn twists, we obtain a new positive factorization of the identity in $\Gamma_2 $, which yields a genus-$2$ Lefschetz fibration of type $(28,1)$. (One can in fact find five more similar subwords and iterate this procedure; see \cite{AkhmedovParkJY}.) Similarly, applying lantern relations to the positive word $W_5$, Endo and Gurtas found the fibrations of the type $(14,3)$  $(16,2)$ \cite{EndoGurtas}. All other pairs, except for $(10,10)$, are realized by the fibrations constructed in the next section.
\end{remark}

\enlargethispage{0.2in}
\begin{remark} \label{nonminimalremark}
Going back to the main theme of our article, we can ask if there is a \emph{nonminimal} genus-$2$ Lefschetz fibration $(X, f)$ where $X$ is an exotic  $\CP \# \, p \, \CPb$ or $3 \, \CP \# \, q \, \CPb$. Let $(X,f)$ be such a fibration of type $(n,s)$, which we can assume to be relatively-minimal after blowing-down any exceptional spheres on the fibers. Let $\kappa(X)$ denote the symplectic Kodaira dimension of $X$. Clearly, $X$ is not a rational or a ruled surface, so there are three possibilities by the work of Sato: (i) $\kappa(X)=0$, $c_1^2(X)=-1$ or $-2$; (ii) $\kappa(X)=1$, $c_1^2(X)=-1$ and $s=1$; (iii) $\kappa(X)=2$ , $c_1^2= -1$ and $s=0$ \, \cite{SatoKodaira}[Theorems 5.5. and 5.12]. We observe that (iii) is not possible, since the only possible types $(20,0)$ and $(30,0)$ have $c_1^2\neq-1$. On the other hand (i) or (ii) are possible only if $X$ is an exotic $\K$ surface blown-up once or twice; the only possible types in this case being $(28,1)$ or $(30,0)$ by similar arguments. Hence one does not get any exotic $4$-manifolds via nonminimal genus-$2$ Lefschetz fibrations, except possibly for those whose minimal models are exotic $\K$ surfaces.
\end{remark}

\smallskip
\section{Small exotic $4$-manifolds from genus-$2$ fibrations} \label{sec:exotic}

Here we will construct minimal symplectic $4$-manifolds which are homeomorphic to small number of blow-ups of $\CP$ and $3\CPb$, starting with the most interesting examples: exotic symplectic rational surfaces fibered by genus-$2$ curves. 

\subsection{Exotic symplectic rational surfaces} \

Using the positive factorizations $W_1$ and $W_2$  of $t_{\delta}$ in $\Gamma_2^1$ given  in Section~\ref{smallLFs} we obtain three positive factorizations of $t_{\delta}^2$: 
\begin{align*}
\widetilde{W}_1 &= W_1^{\phi_1} W_1 =  t_ {\phi_1(e)}     t_{\phi_1(x_1)}  t_{\phi_1(x_2)} t_{\phi_1(x_3)}   t_{\phi_1(d)}  t_{\phi_1(c)}  t_{\phi_1( x_4 )} t_ {e}     t_{x_1}  t_{x_2} t_{x_3}   t_{d}  t_{c}  t_{ x_4 }, \\
\widetilde{W}_2 &= W_1^{\phi_2} W_2 =  t_ {\phi_2(e)}     t_{\phi_2(x_1)}  t_{\phi_2(x_2)} t_{\phi_2(x_3)}   t_{\phi_2(d)}  t_{\phi_2(c)}  t_{\phi_2( x_4 )} (t_{B_0} t_{B_1} t_{B_2} t_{c})^2,  \\
\widetilde{W}_3 &= W_2^{\phi_3} W_2 = ( t_{\phi_3(B_0)} t_{\phi_3(B_1)} t_{\phi_3(B_2)} t_{\phi_3(c)} )^2 (t_{B_0} t_{B_1} t_{B_2} t_{c})^2 . 
\end{align*}

For each $i=1,2,3$, let $(\widetilde{X}_i, \widetilde{f}_i)$ be the Lefschetz fibration prescribed by the positive factorization  $\widetilde{W}_i$, so that we have genus-$2$ Lefschetz fibrations of types of $(8,6)$, $(10,5)$ and $(12,4)$. By Proposition~\ref{minimalityprop}, each one of them is minimal. Moreover, if we assume that $\widetilde{X}_i$ is simply-connected, we calculate $b^-(\widetilde{X}_i) =6+i$ using the equation~$(\ref{b1})$. As already argued in Theorem~\ref{scgeography}, $\widetilde{X}_i$ is then homeomorphic to $\CP \# (6+i) \CPb$, for $i=1, 2, 3$, but certainly not diffeomorphic to it, since the latter are nonminimal. Thus, it  remains to show that for appropriate choices of mapping classes $\phi_i$, we can indeed get $\pi_1(\widetilde{X}_i) =1$. 

Each $\widetilde{W}_i$ contains the non-conjugated word $W_j$, for $j=1$ or $2$. As shown in Section~\ref{smallLFs}, the relations induced by the vanishing cycles prescribed by either $W_j$, when added to the presentation $\pi_1(\Sigma_2) \cong \langle a_1, b_1, a_2, b_2 \, | \, [a_1, b_1][a_2,b_2]=1 \, \rangle$, already give an abelian group generated by $a_2$ and $b_2$. We obtain a presentation for $\pi_1(\widetilde{W}_i)$ by adding further relations induced by the vanishing cycles coming from the conjugated word $W_1^{\phi_1}$, $W_1^{\phi_2}$ or $W_2^{\phi_3}$, which clearly will be an abelian quotient of $\Z^2$. Therefore, it suffices to understand these additional relations in 
the first integral homology group $H_1(\Sigma_2) \cong \Z^4$, the generators of which we will identify with $a_1, b_1, a_2, b_2$. 

The upshot of the above discussion is that we can choose the desired $\phi_i$ based on its action on $H_1(\Sigma_2)$.  Using the Picard-Lefschetz formula we can easily calculate the effect of the Dehn twist $t_{a}$ on the homology class of a curve $b$ as
\[ 
[t_a^m(b)] = [b] + m ( a \cdot b) [ a] \, \in H_1(\Sigma_2) , 
\]
where $a \cdot b$ is the algebraic intersection number for the oriented curves $a$ and $b$.  We will express $\phi_i$ as products of Dehn twists to further simplify our calculations, and we will indeed show that a simple multitwist will work for all! 

We will take $\phi_i=  t^{-1}_{c_1} t_{c_4}$ for each $i=1, 2, 3$. 

Let us first consider $\widetilde{W}_1$. In this case, from the earlier relations~$(\ref{eqn:prestn1})-(\ref{eqn:prestn8})$, we see that the nonconjugated word $W_1$ induces the abelianized relations: 
\begin{equation} 
a_1 + 2a_2 + 4 b_2=0 \, \text{ and  } \, b_1 + b_2= 0 , 
\end{equation}
whereas $W_1^{\phi_1}$ induces the following abelianized relations: 
\begin{equation} \label{conjW1relations} 
 a_1 + b_1 +6 a_2 + 4 b_2=0 \, \text{ and  } \,  b_1 + a_2 +b_2 =0.
\end{equation}
The four relations together imply that $a_2=0$ and $b_2=-b_1=0$. Hence $\pi_1( \widetilde{X}_1) = 1$ for this choice of $\phi_1$. 

For $\widetilde{W}_2$ and $\widetilde{W}_3$, the nonconjugated word $W_2$ induces the abelianized relations:
\begin{equation}
a_1 + a_2= 0 \, \text{ and  } \, b_1 + b_2 =0 
\end{equation}
derived from the relations~$(\ref{eqn:prestn8})$ and $(\ref{eqn:prestn10})$. Now for $\widetilde{W}_2$, the abelianized relations above, induced by $W_2$, together with those induced by  $W_1^{\phi_2}= W_1^{\phi_1}$ given in $(\ref{conjW1relations})$ immediately imply $a_2=0$. In turn, $a_1=0$, so $b_2= b_1 + 2b_2= a_1 + b_1 + 4 a_2 + 2 b_2=0$. We therefore get $\pi_1( \widetilde{X}_2) = 1$ for the chosen  $\phi_2$. 

The calculation for $\widetilde{W}_3$ is even simpler. $W_2^{\phi_3}$ induces the following abelianized relations:
\begin{equation} \label{conjW2relations}
a_1+ b_1+a_2 = 0 \, \text{ and  } \, b_1 + a_2 +b_2 =0 , 
\end{equation}
which, together with those above induced by $W_2$, kill all the generators. Thus $\pi_1( \widetilde{X}_3) = 1$ for $\phi_3$ we have chosen. 

\smallskip
Hence, we have proved 

\begin{theorem}  \label{exoticCP} 
For each $i=1, 2, 3$, 
$(\widetilde{X}_i,\widetilde{f}_i)$ is a minimal, decomposable genus-$2$ Lefschetz fibration, whose total space is an exotic symplectic $\CP \# (6+i) \CPb$. 
\end{theorem}

\smallskip
\begin{remark}
The existence of minimal symplectic $4$-manifolds in the above homeomorphism classes was already established. As shown by Donaldson in his seminal paper \cite{Donaldson87}, for $p=9$ blow-ups of $\CP$, the Dolgachev surface $E(1)_{2,3}$ was the first example; for $p=8$, Kotschick proved that Barlow's surface was minimal \cite{Kotschick}; and for $p=7$, the first example was obtained by Jongil Park's breakthrough article utilizing rational blowdowns \cite{ParkJ}. As we mentioned in Introduction, examples of minimal, decomposable, genus-$2$ Lefschetz fibrations in the homeomorphism class  of $\CP \# 9 \CPb$ were obtained previously by Fintushel and Stern \cite{FSLF}.
\end{remark}

\begin{remark} \label{noncomplex} 
With a small variation of the above constructions as in \cite{OzbagciStipsicz,Korkmaz,BaykurHolomorphic}, we can  also produce minimal genus-$2$ Lefschetz fibration with any prescribed abelian group of rank at most $2$: Let $i=1,2,3,$ and let $m=(m_1,m_2)$, where $m_1$ and $m_2$ are nonnegative integers.
Let $(\widetilde{X}_{i,m},\widetilde{f}_{i,m})$ be the Lefschetz fibration with monodromy $\widetilde{W}_i$ with $\phi_i= t^{-m_1}_{c_1} t_{c_4}^{m_2} $. 
A very similar calculation shows that $\pi_1(\widetilde{X}_{i,m})$ has an abelian presentation
$\langle a_2, b_2 \, | \, m_1 b_2 =0 , \, m_2 a_2=0 \rangle$. 
Hence, we get symplectic $\widetilde{X}_{i,m}$, with $c_1^2(\widetilde{X}_{i,m})=3-i$ and $\chi_h(\widetilde{X}_{i,m})=1$, and any prescribed abelian fundamental group of rank at most $2$.\footnote{Another example with  $c_1^2=1$, $\chi_h=1$ and $\pi_1= \Z_3$ is given in a recent preprint of Akhmedov and Monden using the lantern substitution~\cite{AkhmedovMonden}. Akhmedov pointed out to us that they have now updated their arxiv paper to include some nonexplicit examples with $\pi_1 = \Z \oplus \Z_m$.} 

For each fixed $i=1, 2, 3$, infinitely many of these symplectic fibrations have total spaces which cannot be a complex surface. Here is a quick argument: by the Enrique-Kodaira classification, a complex surface with odd first betti number is either of type VII or elliptic. However, any minimal elliptic surface with odd betti number has Euler characteristic zero, whereas for any minimal type VII surface it is $2$. So any $\widetilde{X}_i$ with odd $b_1(\widetilde{X}_i)$ does not admit a complex structure, and we have infinitely many of them (with $\pi_1 \cong \Z \oplus \ (\Z / \, n \Z)$ for varying $n$).
\end{remark}

\begin{remark}
Very recently, Rana, Tevelev and Urz\'{u}a have given an elaborate proof of the simple-connectivity of the Craighero-Gattazzo  surface \cite{RanaTevelevUrzua}, which was shown earlier by Dolgachev and Werner to carry a holomorphic genus-$2$ Lefschetz fibration \cite{DolgachevWerner}. The authors also showed that the Dolgachev surface $E(1)_{2,3}$ admits a holomorphic genus-$2$ Lefschetz fibration. It would be certainly interesting to know whether these holomorphic fibrations are equivalent to ours, and if there is a holomorphic counterpart of our genus-$2$ Lefschetz fibration with $c_1^2=2$ as well. The same question goes for the genus-$2$ Lefschetz fibrations with $b^+=3$ we construct below. Note that by the work of Siebert and Tian \cite{SiebertTian} any genus-$2$ Lefschetz fibration without reducible fibers (satisfying a mild condition on its monodromy) is holomorphic, but as we observed in the previous section, \emph{every} minimal genus-$2$ Lefschetz fibrations with $b^+ \leq 3$ contains reducible fibers.
\end{remark}

\smallskip
\subsection{Exotic symplectic $4$-manifolds with $b^+=3$} \

We will now construct minimal symplectic genus-$2$ Lefschetz fibrations whose total spaces are homeomorphic to blow-ups of $3 \CP$. This time we will define new positive factorizations of  $t_{\delta}^2$  in $\Gamma_2^1$ using all of  $W_1, W_2, W_3, W_4, W_5$ given in Section~\ref{smallLFs}:

\begin{align*}
\widehat{W}_1 &= W_1^{\phi} W_1 W_1=  t_ {\phi(e)}     t_{\phi(x_1)}  t_{\phi(x_2)} t_{\phi(x_3)}   t_{\phi(d)}  t_{\phi(c)}  t_{\phi( x_4 )} (t_ {e}     t_{x_1}  t_{x_2} t_{x_3}   t_{d}  t_{c}  t_{ x_4 })^2 \\
\widehat{W}_2 &= W_1^{\phi} W_1  W_2 =t_ {\phi(e)}     t_{\phi(x_1)}  t_{\phi(x_2)} t_{\phi(x_3)}   t_{\phi(d)}  t_{\phi(c)}  t_{\phi( x_4 )} t_ {e}     t_{x_1}  t_{x_2} t_{x_3}   t_{d}  t_{c}  t_{ x_4 } (t_{B_0} t_{B_1} t_{B_2} t_{c})^2  \\
\widehat{W}_3&= W_2^{\phi} W_2 W_1 = ( t_{\phi(B_0)} t_{\phi(B_1)} t_{\phi(B_2)} t_{\phi(c)} )^2 (t_{B_0} t_{B_1} t_{B_2} t_{c})^2 t_ {e}     t_{x_1}  t_{x_2} t_{x_3}   t_{d}  t_{c}  t_{ x_4 } \\
\widehat{W}_4&= W_2^{\phi} W_2 W_2 = ( t_{\phi(B_0)} t_{\phi(B_1)} t_{\phi(B_2)} t_{\phi(c)} )^2 (t_{B_0} t_{B_1} t_{B_2} t_{c})^4  \\
\widehat{W}_5&= W_1^{\phi}  W_3 =  t_ {\phi(e)}     t_{\phi(x_1)}  t_{\phi(x_2)} t_{\phi(x_3)}   t_{\phi(d)}  t_{\phi(c)}  t_{\phi( x_4 )} t_ e     t_{x_1}  t_{x_2} t_{x_3}   t_d  (t_{c_1} t_{c_2})^6  t_{ x_4 }  \\
\widehat{W}_6&= W_4  W_1 =  t_{B_0} t_{B_1} t_{B_2} t_{c} t_{B_0} t_{B_1} t_{B_2} (t_{c_1} t_{c_2})^6  t_ {e}     t_{x_1}  t_{x_2} t_{x_3}   t_{d}  t_{c}  t_{ x_4 }  \\
\widehat{W}_7&= W_5  W_1 =  (t_{c_1} t_{c_2} t_{c_3} t_{c_4} t_{c_5} t_{c_5} t_{c_4} t_{c_3} t_{c_2} t_{c_1})^2  t_ {e}     t_{x_1}  t_{x_2} t_{x_3}   t_{d}  t_{c}  t_{ x_4 }  \\
\widehat{W}_8&= W_5  W_2 =  (t_{c_1} t_{c_2} t_{c_3} t_{c_4} t_{c_5} t_{c_5} t_{c_4} t_{c_3} t_{c_2} t_{c_1})^2  (t_{B_0} t_{B_1} t_{B_2} t_{c})^2 .
\end{align*}
For simplicity, we will use the same conjugation $\phi= t^{-1}_{c_1} t_{c_4}$ for $\widehat{W}_i$, $i=1, \ldots, 5$, and indeed \emph{no} conjugation for $i= 6, 7, 8$.

Letting $(\widehat{X}_i, \widehat{f}_i)$ be the Lefschetz fibration prescribed by the positive factorization  $\widehat{W}_i$ for $i=1, \ldots, 8$, we obtain genus-$2$ Lefschetz fibrations of the types $(12,9), (14,8), (16,7), (18,6), (20,5), (22,4), (24,3), (26,2)$. They are all minimal by Proposition~\ref{minimalityprop}, and once we show that  $\widehat{X}_i$ is simply-connected,  we can once use Theorem~\ref{scgeography} again to conclude that $\widehat{X}_i$ is an exotic $3 \CP \# (11+i) \CPb$, for $i=1, \ldots, 8$. 

To prove that $\pi_1(\widehat{X}_i)= 1$, observe that we already have relations induced by a \emph{subcollection} of the vanishing cycles, killing $\pi_1(\Sigma_2)$. This is the case for the subcollection coming from $W_1^{\phi} W_1$, $W_2^{\phi} W_2$, $W_4$ and $W_5$, as we have shown in the previous subsection for the first two words, and in Section~\ref{smallLFs} for the last two. Since the separating curve $c$ is disjoint from 
$c_1, c_4$, we have $t_{\phi(c)}= t_c$. So the vanishing curves in $W_1^{\phi} W_1$ are contained in the vanishing curves of $\widehat{W}_5$. Every  $\widehat{W}_i$ contains one of these subwords, and therefore the corresponding $\widehat{X}_i$ is simply-connected.

\smallskip
We have proved:

\begin{theorem} \label{exotic3CP}
For each $i=1, \ldots, 8$, $(\widehat{X}_i, \widehat{f}_i)$ is a minimal, decomposable genus-$2$ Lefschetz fibration, whose total space is an exotic symplectic $3\CP \# (11+i) \CPb$. 
\end{theorem}

\smallskip
\begin{remark}
As we have noted in Introduction, we expect that there are also minimal, decomposable genus-$2$ Lefschetz fibrations in the homeomorphism class of $3\CP \# 11 \CPb$, which we have not succeeded in constructing yet. On the other hand, the existence of minimal symplectic $4$-manifolds in the  above homeomorphism classes were previously shown by other authors. The first examples for $q=19$ blow-ups of $3 \CP$ were given by Friedman and Morgan \cite{FriedmanMorgan} and Stipsicz and Szabo \cite{StipsiczSzabo}; for $ 14 \leq q \leq 18$ by Gompf, using symplectic fiber sums \cite{Gompf}; and for $q = 12, 13$ by Doug Park using the same technique  \cite{ParkD}.
\end{remark}

\vspace{0.1in}
The explicit construction of $\widehat{X}_i$ moreover  allows us to produce an infinite family of minimal symplectic $4$-manifolds in each homeomorphism class $3\CP \# (11+i) \CPb$, for $i=1, \ldots, 8$, which will follow from the following lemma:

\begin{lemma} 
Let $W, W', W''$ be positive factorizations (of the identity element) in $\Gamma_g$, with $W= W' W''$, 
\label{infinite} and let $X$ be the total space of the Lefschetz fibration prescribed by $W$ with \mbox{$b^+(X)>1$.} Assume that there is a nonseparating simple closed curve $\beta$ on $\Sigma_g$ which is trivial both in $\pi_1(\Sigma_g) / \, N'$ and in $\pi_1(\Sigma_g) / \, N''$, where $N'$ and $N''$ are the subgroups of $\pi_1(\Sigma_g)$ normally generated by the curves in positive factorizations $W'$ and $ W''$, respectively. Then there exist an infinite family of pairwise non-diffeomorphic minimal symplectic $4$-manifolds (resp. irreducible $4$-manifolds which do not admit any symplectic structures) all homeomorphic to $X$. 
\end{lemma}

\begin{proof}
The proof relies on the following illustrious theorem of Fintushel and Stern: if there exists an embedded symplectic torus $T$ in a simply-connected symplectic $4$-manifold $X$ with $b^+(X)>1$ such that  $[T]^2=0$ and $\pi_1(X \ \setminus T) =1$, then one can perform \emph{knot surgery} operation along fibered knots with different Alexander polynomials to produce an infinite family of pairwise non-diffeomorphic symplectic $4$-manifolds in the same homeomorphism class as $X$  \cite{FSKnotsurgery}. Using non-fibered knots with different Alexander polynomials, one can similarly produce an infinite family of irreducible $4$-manifolds none of which can admit symplectic structures. 

Now let $W, W', W''$ be positive factorizations in $\Gamma_g$, with $W= W' W''$, and $\beta$ be a nonseparating curve on $\Sigma_g$ which is trivial both in $\pi_1(\Sigma_g) / \, N'$ and in $\pi_1(\Sigma_g) / \, N''$, for $N', N''$. There exists a nonseparating curve $\alpha$ on $\Sigma_g$ intersecting $\beta$ at one point. Since $W'$, as a mapping class, stabilizes any nonseparating curve on $\Sigma_g^m$, we can take a parallel transport of $\alpha$ over a curve $\gamma$ enclosing the Lefschetz critical values corresponding to $W'$ to produce a Lagrangian torus $T$ fibered over $\gamma$ in the symplectic Lefschetz fibration $(X,f)$ constructed from $W$. It follows from our assumption on $\beta$ that in $X$, the torus $T$ intersects an immersed $2$-sphere $S$, which is composed of two immersed disks bounding the curve $\beta$ viewed on a regular fiber over $\gamma$. Thus, if only such $\beta$ exist, we get a Lagrangian torus $T$ in $X$ with $\pi_1(X \setminus T)=1$ (since by Seifert-Van Kampen, it is normally generated by the meridian $\mu$ of $T$, which in this case bounds an immersed disk) and $[T] \neq 0$ in $H_2(X)$ (since $T \cdot S \neq 0$). By Gompf's trick, any homologically essential Lagrangian becomes symplectic after a local perturbation of the symplectic form  \cite{Gompf}, which necessarily has square zero. This is the symplectic torus $T$ we need in order  to invoke Fintushel and Stern's theorem. 
\end{proof}

\begin{cor} \label{infinitecor}
There exists an infinite family of pairwise non-diffeomorphic minimal symplectic $4$-manifolds   (resp. irreducible $4$-manifolds which do not admit any symplectic structures) in each homeomorphism class $3\CP \# (11+i) \CPb$, for each $i=1, \ldots, 8$.
\end{cor}

\begin{proof}
We claim that the desired $\beta$ is found in abundance in each positive factorization $\widehat{W}_i$, which clearly already satisfies all the other assumptions in Lemma~\ref{infinite}. 

In $\widehat{W}_1$, $\widehat{W}_3$, $\widehat{W}_6$, $\widehat{W}_7$, we can take $\beta$ to be any $x_i$ appearing in the right most $W_1$ factor. Similarly in $\widehat{W}_2$, $\widehat{W}_4$, $\widehat{W}_8$, any $B_i$ in the right most $W_2$ factor would do. Lastly, by taking an even simpler conjugation, $\phi= t_{c_4}$, we can make it easier to spot the desired $\beta$ in $\widehat{W}_5$. Let us first verify that this still results in a simply-connected $\widehat{X}_5$. The relations induced by the vanishing cycles coming from the nonconjugated word $W_3$ already give a presentation of the cyclic group $\Z_2$ generated by $a_2$. The 
group $\pi_1(\widehat{X}_5)$ is calculated by adding further relations coming from the conjugated word $W_1^{\phi}$. So all we need to see is that the latter kills $a_2$. For $\phi= t_{c_4}$, the abelianized relation $b_1 + a_2 +b_2 =0$ coming from $W_1^{\phi}$, together with the $b_1+ b_2=0$ relation coming from $W_1$, implies that $a_2=0$. Thus $\pi_1(\widehat{X}_5)=1$. Since $x_3$ is disjoint from $c_4$, we have $t_{\phi(x_3)}=t_{x_3}$ both in $W_1^{\phi}$ and in $W_1$. So we can take $\beta= x_3$.
\end{proof}

\smallskip
\subsection{Even smaller, but not fibered, exotic $4$-manifolds} \label{evensmaller} \

Relying on our explicit construction of the smallest Lefschetz fibration $(X_1, f_1)$, we can construct minimal symplectic $4$-manifolds in the homeomorphism classes of  $\CP \# p \CPb$ and $3 \CP \#q \CPb$ for even smaller $p$ and $q$. As dictated by Theorem~\ref{scgeography}, this can only be achieved after giving up on the fibered aspect of our previous constructions of course. Below, we will content ourselves with demonstrating our ideas just for a couple examples, namely for $p=4$ and $q=6$, the smallest values we are able to strike. 

We will obtain the desired simply-connected minimal symplectic $4$-manifolds by generalizing our earlier constructions in two ways: first, we will build small Lefschetz fibrations \emph{over $\Sigma_h$} with $h>0$; and second, we will perform Luttinger surgeries along fibered Lagrangian tori which no longer need to preserve the fibration structure.

\smallskip
The notion of a Lefschetz fibration, along with all the associated definitions, can be easily extended to maps $f: X \to \Sigma_h$ for any $h \geq 0$. In this case, a positive factorization also contains a product of $h$ commutators $[\mathcal{A}_i, \mathcal{B}_i]$ in $\Gamma_g^m$, that is:
\begin{equation} \label{factorizationcommutators}
[\mathcal{A}_1, \mathcal{B}_1]\cdots [\mathcal{A}_h, \mathcal{B}_h]\cdot t_{c_l} \cdots t_{c_2} t_{c_1} = t_{\delta_1}^{k_1} \cdots t_{\delta_m}^{k_m} ,
\end{equation}
which prescribes a \emph{symplectic} genus-$g$ Lefschetz fibration $(X,f)$ \emph{over $\Sigma_h$} when $g \geq 2$ \cite{GompfStipsicz, BaykurKorkmazMonden}.  The Euler characteristic of $X$ is given by $\eu(X)=(2-2g)(2-2h)+l$, and when $g=2$, we have the identical signature formula $\sigma(X)= -\frac{1}{5}(3n+s)$, where $n$  and $s$ are the number of nonseparating and separating $\{c_i\}$. As shown by Stipsicz,  when $h>0$, an additional perk is that $X$ is always minimal \cite{StipsiczChern}. However in this case $X$ can never be simply-connected, since $\pi_1(X)$ surjects onto  $\pi_1(\Sigma_h)$.

Clearly, we can add any commutator $[\mathcal{A}_i, \mathcal{B}_i]=1$ in $\Gamma_g^m$ to a positive factorization of a non-minimal Lefschetz fibration $(X, f)$, and pass to a positive factorization of the type~$(\ref{factorizationcommutators})$ of a decomposable, minimal $(X',f')$ over a positive genus surface. It is easy to see that adding $h$ many copies of the trivial commutator, i.e. when $\mathcal{A}_i= \mathcal{B}_i =1$, we moreover have
\[ \pi_1(X') \cong \pi_1(\Sigma_h) \oplus  (\pi_1(\Sigma_g) \, / \, N(c_1, \ldots, c_l) \,) \, , \]
assuming $m>0$. 

Now, adding one trivial commutator to the smallest positive factorization $W_1$ we have constructed, we get a symplectic genus-$2$ Lefschetz fibration $(X_1', f_1')$ over $T^2$ with a section $S'$. We have $\eu(X'_1)=7$, $\sigma(X_1')=-3$, and $\pi_1(X_1') \cong \Z^4$. 

By a slight abuse of notation, we also denote by $a_1, b_1, a_2, b_2$ the curves on a subsurface $\Sigma_2^1$ of a regular fiber $F \cong \Sigma_2$ representing the standard generators of $\pi_1(\Sigma_2^1)$, but isotoped to be in minimally intersecting position. In the same fashion, we pick $a, b$ on the base $T^2$, which also denote the standard generators of $\pi_1(T^2)$, while making sure they avoid  a disk $D \subset T^2$ containing $\text{Crit}(f)$. Lastly let $a'_1$ be a parallel copy of $a_1$ on $\Sigma_2^1$ and $b'$ be a parallel copy of $b$ on $T^2 \setminus D$. We can then assume that the parallel transport of any $\alpha \in \{a_1, b_1, a_2, b_2, a'_1 \}$ over any $\gamma \in \{a, b, b' \}$ is a Lagrangian torus $T$ fibered over $\gamma$. Through the trivialization $X_1 \setminus (f^{-1}(D) \cup \nu S) \cong \Sigma_2^1 \x \Sigma_1^1$, we can view $T$ as a Lagrangian $\alpha \x \gamma$ with respect to a product symplectic form on  $\Sigma_2^1 \x \Sigma_1^1$. 


We claim that performing the following four disjoint Luttinger surgeries in $X'_1$:
\[ (a_2 \x b, a_2, 1),\, (b_2 \x b', b_2, 1), \, (a_1 \x a, a, 1),\, (a'_1 \x b, b, 1),\]
we obtain a simply-connected $4$-manifold $\widetilde{X'_1}$. Here, we encode the surgery data by the triple $(T, \lambda, k)$, as in \cite{ABP, FSReverseEngineering}. Per the choices we made above, we can appeal to the meticulous work of Baldridge and Kirk in \cite{BK} to deduce that $\pi_1(\widetilde{X'_1})$ has a presentation with generators $a_i, b_i, a, b$, where the following relations hold (among many others we do not include): 
\begin{equation*}
 a_1 a_2^2b_2^2=1,  \, b_1 b_2=1, [b_2, b]=1, \\
\end{equation*}
\begin{equation*}
\mu_1 a_2 = 1,  \, \mu_2 \, b_2 =1 , \, \mu_3 \, a =1, \, \mu_4 \, b = 1 , \\
\end{equation*}

\noindent where $\mu_i$ are the meridians of the surgered Lagrangian tori, each given by some commutator of the pairs $\{b_2, a\}$, $\{a_2, a\}$, $\{b_1, b\}$ and $\{b_1, a\}$.\footnote{One can certainly determine these commutators on the nose, but there will be no need for our calculations here.} Since $b_1 b_2=1$ and $[b_2, b]=1$, it follows that $b_1$ and $b$ commute. So from $\mu_3 a=1$, we get $a=1$. For all other $\mu_i$ are commutators of $a$, the remaining relations in the last line above imply $b=a_2=b_2=1$, and in turn $a_1=b_1=1$. Thus $\pi_1(\widetilde{X'_1})=1$.

Clearly $\eu(\widetilde{X'_1}) = \eu(X'_1)$ and $\sigma(\widetilde{X'_1})= \sigma(X'_1)$. Since any exceptional sphere can be isotoped away from Lagrangian tori \cite{Welschinger}, and   for each Luttinger surgery there is an inverse Luttinger surgery, minimality of $X'_1$ implies that 
$\widetilde{X'_1}$ is minimal. Hence $(\widetilde{X'_1}$  is a minimal symplectic $4$-manifold which is an exotic $\CP \# \, 4 \CPb$.

Our construction of a minimal symplectic $3\CP \# \, 6 \CPb$ can be built on the same construction scheme: now add \emph{two} trivial commutators to the positive factorization $W_1$ to obtain a symplectic genus-$2$ Lefschetz fibration $(X''_1, f''_1)$ over $\Sigma_2$ with a section $S''$. Here we calculate $\eu(X''_1)=11$, $\sigma(X''_1)=-3$, $\pi_1(X''_1) \cong \pi_1(\Sigma_2) \oplus \Z^2$. Choosing $a_1, b_1, a_2, b_2, a'_1$ on the fibers as before, and the curves $a, b, b', A, B, B'$ in the same fashion on the base $\Sigma_2$, away from the critical values, we can define Lagrangian tori $\alpha \x \gamma$ for any $\alpha \in \{a_1, b_1, a_2, b_2, a'_1\}$ and $\gamma \in \{a, b, b', A, B, B'\}$. Now let $\widehat{X''_1}$ be the result of the following disjoint Luttinger surgeries:
\[ (a_2 \x b, a_2, 1),\, (b_2 \x b', b_2, 1), (a_1 \x a, a, 1),\, (a'_1 \x b, b, 1), \, (a_1 \x A, A, 1),\, (a'_1 \x B, B, 1). \]
Then we get a presentation for $\pi_1(\widehat{X''_1})$ with generators $a_i, b_i, a, b, A, B$, where the following relations hold (again, among many others we will not need): 
\begin{equation*}
 a_1 a_2^2b_2^2=1,  \, b_1 b_2=1, [b_2, b]=1, [b_2, B]=1 , \\
\end{equation*}
\begin{equation*}
\mu_1 a_2 = 1,  \, \mu_2 \, b_2 =1 , \, \mu_3 \, a =1, \, \mu_4 \, b = 1 , 
\mu_5 A = 1,  \, \mu_6 \, B =1 , \\
\end{equation*}

\noindent where the additional meridians $\mu_5, \mu_6$ are some commutators of the pairs $\{b_1, B\}$ and $\{b_1, A\}$. By arguments identical to those we had before (now repeating them also for $A$ and $B$), we observe that these relations suffice to kill all the generators. We thus have a simply-connected minimal symplectic $\widehat{X''_1}$, which is an exotic $3 \CP \# 6 \CPb$.

\smallskip
We summarize our results in the following:

\begin{theorem} \label{thm:evensmaller}
Luttinger surgeries along fibered Lagrangian tori in the symplectic genus-$2$ Lefschetz fibrations $(X'_1, f'_1)$ over $T^2$ and $(X''_1, f''_1)$ over $\Sigma_2$ yield minimal symplectic $4$-manifolds $\widetilde{X'_1}$ and $\widehat{X''_1}$ homeomorphic to $\CP \# 4 \CPb$ and $3 \CP \# 6 \CPb$, respectively. 
\end{theorem}

\smallskip
\begin{remark}
We can easily use different genus-$2$ Lefschetz fibrations over $T^2$ and $\Sigma_2$ in our construction by adding other commutators to the positive factorization $W_1$ instead of the trivial one. For instance, if we add $[\phi, 1]= [t^{-1}_{c_1} t_{c_4}, 1]$ instead, we can 
produce a simply-connected $4$-manifold without the first two Luttinger surgeries in each case. 

On this note, recall from Section~\ref{preliminaries} that conjugating the first factors $W_1$ or $W_2$  with $\phi=t^{-1}_{c_1} t_{c_4}$  in the positive factorizations of \, $\widetilde{W}_i, \widehat{W}_i$ we had earlier amount to fibered Luttinger surgeries. Thus, several of the simply-connected minimal genus-$2$ Lefschetz fibrations $(\widetilde{X}_i, \widetilde{f}_i)$ and $(\widehat{X}_i, \widehat{f}_i)$ we have constructed earlier could be obtained from genus-$2$ Lefschetz fibrations corresponding to nonconjugated products via a pair of fibered Luttinger surgeries along $b_1 \x \gamma$ and $a_2 \x \gamma$, where $\gamma$ is a curve on the base enclosing all the critical values coming from the first factor $W_1$ or $W_2$. 
\end{remark}

\begin{remark} \label{experts}
A few words for the experts:  in a nutshell, the manifolds $\widetilde{X'_1}$ and $\widehat{X''_1}$ can be seen to be produced by first taking a symplectic sum of $(S^2 \x T^2) \, \# 3 \CPb$ and $\Sigma_2 \x T^2$, and then performing Luttinger surgeries along Lagrangian tori in the latter summand. This has been a well-exploited approach to successfully construct many small $4$-manifolds (e.g \cite{ABP, AP, BK, FSReverseEngineering}), where the true hardship had been in finding the desired symplectic surfaces with well-understood fundamental group complements, and then obtaining a manageable fundamental group presentation for the resulting $4$-manifold to argue that it is simply-connected. We believe these issues are significantly easier to handle (compare e.g. with the examples of \cite{AP}) when we use the  Lefschetz fibration $(X_1, f_1)$ as above. Similarly, describing a handle decomposition of our $\widetilde{X'_1}$ and $\widehat{X''_1}$ is a relatively easier task departing from the handle decompositions of $(X'_1, f'_1)$, $(X''_1, f''_1)$. 
\end{remark}

\vspace{0.2in}

\end{document}